\documentclass[a4paper,twoside]{article}      
\usepackage{amsmath,amssymb,amsfonts,amsthm,amscd} 
\usepackage{graphics}                 
\usepackage{color}                    
\usepackage{ mathrsfs }
\usepackage{indentfirst}
\usepackage{bbold}
\usepackage{enumerate}
\usepackage{url}         
\usepackage{colonequals} 
\usepackage{a4wide}
\usepackage{mathdots}

\newmuskip\halfthinmuskip
\def\.{\mskip\halfthinmuskip}

\newtheorem{theorem}{Theorem}[section]

\newtheorem{lemma}[theorem]{Lemma}
\theoremstyle{definition}
\newtheorem{remark}[theorem]{Remark}

\def\div{\mathop{\mathrm{div}}\nolimits}
\def\argmin{\mathop{\mathrm{argmin}}}
\def\Prob{\mathop{\mathrm{Prob}}}
\providecommand{\abs}[1]{\left|#1\right|}
\newcommand{\Keywords}[1]{\par\indent
{\small{\textbf{Key words and phrases.} \/} #1}}

\def\A{\mathcal{ A}}
\def\R{\mathbf{ R}}

\def\H{\mathcal{ H}}
\def\P{\mathcal{P}}
\def\E{\mathcal{E}}

\def\F{\mathcal{F}}

\def\H{\mathcal{H}}

\def\u{\textbf{u}}

\def\un#1{\mathrm{#1}}

\def\kT{\beta^{-1}}
\def\txi{\widetilde{\xi}}
\def\oxi{\overline \xi}
\let\e\varepsilon
\def\PSopt{P^*_{\mathrm{opt}}}

\def\Xint#1{\mathchoice
   {\XXint\displaystyle\textstyle{#1}}%
   {\XXint\textstyle\scriptstyle{#1}}%
   {\XXint\scriptstyle\scriptscriptstyle{#1}}%
   {\XXint\scriptscriptstyle\scriptscriptstyle{#1}}%
   \!\int}
\def\XXint#1#2#3{{\setbox0=\hbox{$#1{#2#3}{\int}$}
     \vcenter{\hbox{$#2#3$}}\kern-.5\wd0}}
\def\dashint{\Xint-}

\newlength{\boxwidth}
\date {\today}
\title{Conservative-dissipative approximation schemes for a generalized Kramers  equation}
\author{Manh Hong Duong$^1$ \and Mark A. Peletier$^{2,3}$ \and Johannes Zimmer$^1$}
\begin{document}
\maketitle
\begin{abstract}
We propose three new discrete variational schemes that capture the conservative-dissipative structure of a generalized Kramers equation. The first two schemes are single-step minimization schemes while the third one combines a streaming and a minimization step. The cost functionals in the schemes are inspired by the rate functional in the Freidlin-Wentzell theory of large deviations for the underlying stochastic system. We prove that all three schemes converge to the solution of the generalized Kramers equation.
\end{abstract}

\footnotetext[1]{Department of Mathematical Sciences, University of Bath}
\footnotetext[2]{Department of Mathematics and Computer Sciences, Technische Universiteit Eindhoven}
\footnotetext[3]{Institute for Complex Molecular Systems, Technische Universiteit Eindhoven}

\section{Introduction}

\subsection{The Kramers equation}

In this paper we discuss the variational structure of a generalized \emph{Kramers equation},
\begin{equation}
\label{KRequation}
\partial_t \rho = - \div_q \rho\frac pm + \div_p \rho \nabla_q V+ \gamma\div_p \rho\nabla _p F + \gamma kT \Delta_p \rho, \qquad\text{in } \R^{2d}\times \R^+,
\end{equation}
which is the Fokker-Planck or Forward Kolmogorov equation of the stochastic differential equation
\begin{subequations}
\label{eq:SDE}
\begin{align}
dQ(t)&=\frac{P(t)}{m}dt\label{sto1},
\\dP(t)&=-\nabla V(Q(t))dt -\gamma \nabla F(P(t))dt + \sqrt{2\gamma kT}\, dW(t)\label{sto2}.
\end{align}
\end{subequations}
The system~\eqref{eq:SDE} describes the movement of a particle at position $Q$ and with momentum $P$ under the influence of three forces. One force is the derivative $-\nabla V$ of a background potential $V=V(Q)$, the second is a friction force $-\gamma \nabla F(P)$, and the third is a stochastic perturbation generated by a Wiener process~$W$. The parameter $m>0$ is the mass of the particle (so that the velocity is $P/m$), $\gamma$ is a friction parameter, $k$ is the Boltzmann constant, and $T$ is the temperature of the noise. A common choice for $F$ is $F(P) = P^2/2m$, which results in a linear friction force.

For a stochastic particle given by~\eqref{eq:SDE}, $\rho=\rho(t,q,p)$ characterizes the probability of finding the particle at time~$t$ at position $q$ and with momentum $p$. Equation~\eqref{KRequation} characterizes the evolution of this probability density over time. The three deterministic drift terms in~\eqref{eq:SDE} lead to convection terms in~\eqref{KRequation}, and the noise results in the final term in~\eqref{KRequation}. We use the notation $\div_q$ and similar to indicate that the differential operator acts only on one variable.

Both equations describe the behaviour of a \emph{Brownian particle with inertia}~\cite{Brown}, such as which is large enough to be distinguished from the molecules in the surrounding solvent, but small enough to show random behaviour arising from collisions with those same molecules. Both the friction force and the noise term arise from collisions with the solvent, and the parameter~$\gamma$ characterizes the intensity of these collisions. The parameter $kT$ measures the mean kinetic energy of the solvent molecules, and therefore characterizes the magnitude of the collision noise.  A major application of this system is as a simplified model for chemical reactions, and it is in this context that Kramers originally introduced it~\cite{Kramers40}.

\medskip

The aim of this paper is to discuss variational formulations for equation~\eqref{KRequation}. The theory of such variational structures took off with the introduction of Wasserstein gradient flows by~\cite{JordanKinderlehrerOtto97,JKO98} and of the energetic approach to rate-independent processes~\cite{MielkeTheilLevitas02,Mielke05a}. Both have changed the theory of evolution equations in many ways. If a given evolution equation has such a variational structure, then this property gives strong restrictions on the type of behaviour of such a system, provides general methods for proving well-posedness~\cite{AGS08} and characterizing large-time behaviour~(e.g.,~\cite{CarrilloMcCannVillani03}), gives rise to natural numerical discretizations (e.g.,~\cite{DuringMatthesMilisic10}), and creates handles for the analysis of singular limits (e.g.,~\cite{SandierSerfaty04, Stefanelli08, ArnrichMielkePeletierSavareVeneroni11TR}). Because of this wide range of tools,  the study of variational structure has important consequences for the analysis of an evolution equation.

\begin{remark}
A brief word about dimensions. We make the unusual choice of preserving the dimensional form of the equations, because the explicit constants help in identifying the modelling origin and roles of the different terms and effects, and these aspects are central to this paper. Therefore $Q$ and $q$ are expressed in $\un m$, $P$ and $p$ in $\un{kg\, m/s}$, $m$ in $\un{kg}$, $V$, $F$, and $kT$ in $\un J$, and $\gamma$ in $\un{kg/s}$. The density~$\rho$ has dimensions such that $\int\rho$ is dimensionless. This setup implies that the Wiener process has dimension~$\sqrt \un s$, in accordance with the formal property $dW^2 = dt$.
\end{remark}

\subsection{Variational evolution}

To avoid confusion between the Boltzmann constant and the integer $k$, from now on we define $\kT\colonequals kT$. The authors of \cite{JKO98} studied an equation that can be seen as a simpler, spatially homogeneous case of~\eqref{KRequation}, where $\rho = \rho(t,p)$:
\begin{equation}
\label{eq:ConvDiff}
\partial_t\rho=\gamma \kT\Delta_p\rho + \gamma \div_p\rho\nabla_p F.
\end{equation}
They showed that this equation is a gradient flow of the free energy
\[
\mathcal A_p(\rho) := \int_{\R^{d}} \Bigl[\rho F + \kT\rho\log \rho\Bigr]\, dp
\]
with respect to the Wasserstein metric. This statement can be made precise in a variety of different ways (see~\cite{AGS08} for a thorough treatment of this subject); for the purpose of this paper the most useful one is that the solution $t\mapsto \rho(t,p)$ can be approximated by the time-discrete sequence~$\rho_k$ defined recursively by
\begin{equation}\label{eq:JKOformulation}
\rho_k\in \argmin_{\rho} K_h(\rho,\rho_{k-1}), ~~~ K_h(\rho,\rho_{k-1})\colonequals \frac{1}{2 h} \frac1\gamma d(\rho, \rho_{k-1})^2+ \mathcal A_p(\rho).
\end{equation}
Here $d$ is the Wasserstein distance between two probability measures  $\rho_0(x)dx$ and $\rho(y)dy$ with finite second moment,
\begin{equation*}
d(\rho_0,\rho)^2\colonequals\inf_{P\in\Gamma(\rho_0,\rho)}\int_{\R^{d}\times\R^d}\! \abs{x-y}^2 P(dxdy),
\end{equation*}
where $\Gamma(\rho_0,\rho)$ is the set of all probability measures on $\R^{d}\times\R^d$ with  marginals $\rho_0$ and~$\rho$,
\begin{equation}\label{eq:coupling}
\Gamma(\rho_0,\rho)=\{P\in \P(\R^{d}\times \R^d): P(A\times \R^{d})=\rho_0(A), P(\R^{d}\times A)=\rho(A) \text{ for all Borel subsets } A\subset \R^{d}\}.
\end{equation}
A consequence of this gradient-flow structure is that $\mathcal A_p$ decreases along solutions of~\eqref{eq:ConvDiff}.

\smallskip

Unfortunately, a convincing generalization of this gradient-flow concept and corresponding theory to equations such as the Kramers equation is still lacking. This is related to the mixture of both dissipative and conservative effects in these equations, which we now explain.

\subsection{A combination of conservative and dissipative effects}
\label{subsec:combination}
The full Kramers equation~\eqref{KRequation} is a mixture of  the dissipative behaviour described by~\eqref{eq:ConvDiff} and a Hamiltonian, conservative behaviour.  The conservative behaviour can be recognized by setting $\gamma=0$, thus discarding the last two terms in~\eqref{eq:SDE}; what remains in~\eqref{eq:SDE} is a deterministic Hamiltonian system with Hamiltonian energy $H(q,p) = p^2/2m + V(q)$. The evolution of this system is reversible and conserves~$H$. Correspondingly, the evolution of~\eqref{KRequation} with $\gamma=0$ also is reversible and conserves the expectation of $H$,
\[
\mathcal H(\rho) := \int_{\R^{2d}} \rho(q,p)H(q,p) \, dqdp.
\]

On the other hand, as suggested by the discussion in the previous section, the $\gamma$-dependent terms represent dissipative effects. In the variational schemes that we define below, a central role is played by the $(q,p)$-dependent analogue of $\mathcal A_p$,
\[
\mathcal A(\rho) := \int_{\R^{2d}} \Bigl[\rho(q,p) F(p) + \kT\rho(q,p)\log \rho(q,p)\Bigr]\, dqdp.
\]
Because of the special structure of~\eqref{KRequation}, the functional $\mathcal A$ does not decrease along solutions, but in the particular case $F(p) := p^2/2m$, a `total free energy' functional does: setting
\[
\mathcal E(\rho) \colonequals \mathcal A(\rho) + \int \rho V\, dqdp
= \int \Bigl[H + \kT\log \rho\Bigr]\rho\, dqdp,
\]
we calculate that
\begin{equation}
\label{eq:energydissipation}
\partial_t \mathcal E(\rho(t)) = -\gamma\int_{\R^{2d}}\frac1{\rho(t,q,p)} \Bigl|\rho(t,q,p)\frac pm+ \kT \nabla_p\rho(t,q,p)\Bigr|^2\, dqdp \leq 0.
\end{equation}
The choice $F(p)= p^2/2m$ is related to the fluctuation-dissipation theorem, and we comment on this in Section~\ref{subsec:discussion}.


Because of the conservative, Hamiltonian terms, equation~\eqref{KRequation} is not a gradient flow, and an approach such as~\cite{JKO98} is not possible. In 2000 Huang~\cite{Hua00} proposed a variational scheme that is inspired by~\cite{JKO98}, but modified to account for the conservative effects, and in this paper we describe three more variational schemes for the same equation.

\subsection{Huang's discrete schemes for the Kramers equation}

The time-discrete variational schemes of Huang's and of this paper are best understood through the connection between \emph{gradient flows} on one hand and \emph{large deviations} on the other.
We have recently shown this connection for a number of systems~\cite{ADPZ11,Peletier2011,Laschos2011,DuongLaschosRenger12TR}, including~\eqref{eq:ConvDiff}.

The philosophy can be formulated in a number of ways, and here we choose a perspective based on the behaviour of a single particle.
We start with the simpler case of equation~\eqref{eq:ConvDiff} and the discrete  approximation~\eqref{eq:JKOformulation}.
Let $\{X_\epsilon\}_{\epsilon>0}$ be a rescaled $d$-dimensional Wiener process,
\begin{equation}
\label{eq:Xepsilon}
dX_\epsilon(t)=\sqrt{2\sigma\epsilon}\,dW(t),
\end{equation}
where $\sigma$ is a mobility coefficient.
If we fix $h>0$, then  by Schilder's theorem (e.g.~\cite[Th.~5.2.3]{DemboZeitouni98}), the process $\{X_\epsilon(t): t\in [0,h]\}$ satisfies a large-deviation principle
\[
\Prob\bigl(X_\epsilon(\cdot) \approx \xi(\cdot) \bigr) \sim \exp \Bigl[-\frac1\epsilon I(\xi)\Bigr],\qquad \text{as }\epsilon\to0,
\]
where the \emph{rate functional} $I\colon C([0,h];\R^d)\rightarrow \R\cup \{+\infty\}$ is given by
\begin{equation*}
I(\xi)=\frac{1}{4\sigma}\int_0^h\abs{\dot{\xi}(t)}^2\,dt.
\end{equation*}
The Wasserstein cost function $|x-y|^2$ can be written in terms of $I$ as
\begin{equation}\label{eq:minimalvelocity2}
 |x-y|^2 \;=\; 4h\sigma\,\inf \left\{I(\xi): \xi\in C^1([0,h],\R^d) ~~ \text{such that}~~ \xi(0)=x,\ \xi(h)=y\right\}.
\end{equation}
Hence the cost $|x-y|^2$ can be interpreted as the the probability that a Brownian particle goes from $x$ to $y$ in time $h$, in the sense of large deviations, and rescaled as to be independent of the magnitude of the noise $\sigma$.

The results of~\cite{ADPZ11,Peletier2011,DuongLaschosRenger12TR} concern a similar large-deviation analysis, but now for the empirical measure of a large number $n$ of particles. For this system the limit $n\to\infty$ plays a role similar to $\epsilon\to0$ in the example above. In~\cite{ADPZ11,Peletier2011,DuongLaschosRenger12TR}, it is shown that this rate functional is very similar to the right-hand side of~\eqref{eq:JKOformulation} in the limit $h\to0$. This result explains the strong connection between large deviations on one hand and the gradient-flow structure on the other.

However, the core of the argument of~\cite{ADPZ11,Peletier2011,DuongLaschosRenger12TR} is contained in the Schilder example~\eqref{eq:Xepsilon} and its connection~\eqref{eq:minimalvelocity2} to the Wasserstein cost. Hence we use this simpler point of view to generalize the approximation scheme~\eqref{eq:JKOformulation} to the Kramers equation. There are two different ways of doing this.

\medskip
\noindent\textbf{Approach 1 \cite{Hua00}.}
Instead of the inertia-less Brownian particle given by~\eqref{eq:Xepsilon}, we consider a particle with inertia satisfying
\begin{subequations}
\label{eq:SDE-Huang}
\begin{align}
dQ_\epsilon(t)&=\frac{P_\epsilon(t)}{m}\,dt,
\label{eq:SDE-Huang-A}
\\dP_\epsilon(t)&=\sqrt{2\epsilon\gamma \kT}\, dW(t),
\end{align}
\end{subequations}
which can formally also be written as
\[
m\frac {d^2}{dt^2} Q_\epsilon(t) = \sqrt{2\gamma \kT\epsilon} \, \frac{dW}{dt}(t).
\]
By the Freidlin-Wentzell theorem (e.g.~\cite[Th.~5.6.3]{DemboZeitouni98}), the process $Q_\epsilon(t)$ satisfies a similar large-deviation principle with rate functional $\overline I\colon C([0,h],\R^{d})\rightarrow \R\cup \{+\infty\}$  given by
\begin{equation*}
\overline I(\xi)=\frac{1}{4\gamma \kT}\int_0^h\abs{m\ddot{\xi}(t)}^2\,dt.
\end{equation*}
The comparison with~\eqref{eq:minimalvelocity2} suggests to define a cost functional $\overline{C}_h(q,p;q',p')$  in a similar way, i.e.,
\begin{align}
\overline{C}_h(q,p;q',p')&\colonequals h \;\inf \bigg\{\int_0^h\abs{m\ddot{\xi}(t)}^2\,dt: \xi\in C^1([0,h],\R^d) \quad \text{such that}\notag\\
&\hskip3cm (\xi,m\dot\xi)(0)=(q,p),\ (\xi,m\dot\xi)(h)=(q',p')\bigg\}\notag
\\
&=\abs{p'-p}^2
+12\left|\frac{m}{h}(q'-q)-\frac{p'+p}{2}\right|^2.
\label{def:Chbar}
\end{align}
The second formula follows from an explicit calculation of the minimizer. As above, the interpretation is that of the probabilistic `cost', that is, the large-deviations characterization of the probability of a path of~\eqref{eq:SDE-Huang}  connecting $(q,p)$ to $(q',p')$ over time $h$.
Note that $\overline C_h$ is not a metric, since it is not symmetric, and also $\overline C_h(q,p;q,p)=12|p|^2$ generally does not vanish. Therefore the Wasserstein `distance' $\overline W_h$ defined with $\overline C_h$ as cost is not a metric, but only an optimal-transport cost (see~\cite{Vil03} for an exposition on the theory of optimal transportation).

\medskip

Huang then defines the approximation scheme as

\medskip\noindent
\framebox{
\begin{minipage}{\boxwidth}
\noindent\textbf{Scheme 1~\cite{Hua00}}.
Given a previous state $\rho_{k-1}$, define $\rho_k$ as the  solution of the minimization problem
\begin{equation}
\label{scheme:Huang}
\min_\rho \;\frac1{2h}\frac1\gamma \overline W_h(\rho_{k-1},\rho) + \mathcal A(\rho) + \frac{2m}{\gamma h} \int_{\R^{2d} }\rho(q,p)V(q)\, dqdp,
\end{equation}
where $\overline W_h$ is the optimal-transport cost on $\R^{2d}$ with cost function~$\overline C_h$.
\end{minipage}
}

Huang proves~\cite{Hua00,Hua11} that the approximations generated by this scheme indeed converge to the solution of~\eqref{KRequation} as $h\to0$.

\subsection{Criticism}

Although Scheme 1 is approximately of similar form to~\eqref{eq:JKOformulation}, there are in fact important issues with this scheme:
\begin{enumerate}
\item In~\eqref{KRequation}, the dissipative effects are represented by the terms prefixed by $\gamma$, and the conservative effects by the the Hamiltonian terms $\div_q \rho p/m$ and $\div_p\rho\nabla V$. It would be natural to see these effects play separate roles in the variational formulation. However, in Scheme 1 the effects are mixed, since the final term in~\eqref{scheme:Huang} mixes conservative effects (represented by $V$ and $m$) with dissipative effects (the prefactor $\gamma$, and the role as driving force in a gradient-flow-type minimization).
\item The dependence on $h$ of the final term in~\eqref{scheme:Huang} adds to the confusion; since this parameter is an approximation parameter chosen independently from the actual system, the combination $\mathcal A + 2m/\gamma h \int \rho V$ can not be considered a single driving potential.
\item In fact, in the standard case $F(p)=p^2/2m$ the sum of $\mathcal A$ and $\int\rho V$ is a natural object, since it represents total free energy and decreases along solutions (see Section~\ref{subsec:combination}). Note how the coefficient in this sum is $1$ instead of $2m/\gamma h$.
\end{enumerate}
The way in which $V$ appears in Scheme 1 can be traced back to the fact that of the two conservative terms in~\eqref{KRequation} and~\eqref{eq:SDE}, only $P/m$ is represented in the definition of the cost $\overline C_h$, in the right-hand side of~\eqref{eq:SDE-Huang-A}; the term  $\nabla V$ is missing in~\eqref{eq:SDE-Huang}. Therefore the scheme has to compensate for the other term $\nabla V$ in a different manner.

These arguments lead us to pose the following question, which is the central topic of this paper:
\begin{quote}
\textbf{Can we construct an approximation scheme that respects the conservative-dissipative split?}
\end{quote}
The answer is `yes', and in the rest of this paper we explain how; in fact we detail three different schemes, corresponding to different ways of answering this question.

\subsection{The schemes of this paper}

We take a different approach than Huang did.

\noindent\textbf{Approach 2.}
To set up a new cost functional, we first return to the single-particle point of view, as in~\eqref{eq:Xepsilon} and~\eqref{eq:SDE-Huang}. We now take a particle whose behaviour is a combination of the two Hamiltonian terms in~\eqref{eq:SDE} and a noise term:
\begin{subequations}
\label{eq:SDE-scheme2}
\begin{align}
dQ_\epsilon(t)&=\frac{P_\epsilon(t)}{m}\,dt,\\
dP_\epsilon(t)&=-\nabla V(Q_\epsilon)\, dt + \sqrt{2\gamma \kT\epsilon}\, dW(t),
\label{eq:SDE-scheme2-B}
\end{align}
\end{subequations}
which again can  formally be written as
\[
m\frac {d^2}{dt^2} Q_\epsilon(t) + \nabla V(Q_\epsilon(t)) = \sqrt{2\gamma \kT\epsilon} \, \frac{dW}{dt}(t).
\]
Note how this system differs from~\eqref{eq:SDE-Huang} by the term involving $\nabla V$ in~\eqref{eq:SDE-scheme2-B}.

A very similar application of the Freidlin-Wentzell theorem states that $Q_\epsilon$ satisfies a large-deviation principle as $\epsilon\to0$ with rate function
\[
\widetilde I(\xi)=\frac{1}{4\gamma \kT}\int_0^h\bigl|m\ddot{\xi}(t)+ \nabla V(\xi(t))\bigr|^2\,dt.
\]
This leads to the following scheme.

\medskip
\noindent
\framebox{
\begin{minipage}{\boxwidth}
\noindent\textbf{Scheme 2a.}
We define the cost to be
\begin{multline}
\label{def:widetildeCh}
\widetilde{C}_h(q,p;q',p')\colonequals h\inf \bigg\{\int_0^h\bigl|m\ddot{\xi}(t)+ \nabla V(\xi(t))\bigr|^2\,dt: \xi\in C^1([0,h],\R^d) ~~ \text{such that}\\~~ (\xi,m\dot\xi)(0)=(q,p),\ (\xi,m\dot\xi)(h)=(q',p')\bigg\}.
\end{multline}
Given a previous state $\rho_{k-1}$, define $\rho_k$ as the  solution of the minimization problem
\begin{equation}
\label{scheme:2a}
\min_\rho \;\frac1{2h}\frac1\gamma \widetilde W_h(\rho_{k-1},\rho) + \mathcal A(\rho),
\end{equation}
where $\widetilde W_h$ is the optimal-transport cost on $\R^{2d}$ with cost function~$\widetilde C_h$.
\end{minipage}
}

\medskip

Note how now the term involving $V$ has disappeared from the minimization problem~\eqref{scheme:2a}. In Sections~\ref{section:EulerLagrangEquation}--\ref{section:Convergence} we show that this approximation scheme converges to the solution of~\eqref{KRequation} as $h\to0$.

\bigskip

For practical purposes it is inconvenient that the cost $\widetilde C_h$ in~\eqref{def:widetildeCh} has no explicit expression. It turns out that we may approximate $\widetilde C_h$ with an explicit expression and obtain the same limiting behaviour.

\medskip\noindent
\framebox{
\begin{minipage}{\boxwidth}
\noindent\textbf{Scheme 2b.}
Define
\begin{align}
\widehat{C}_h(q,p;q',p')&\colonequals h\inf\left\{\int_0^h\abs{m\ddot{\xi}(t)+\nabla V(q)}^2\,dt:~~ (\xi,m\dot{\xi})(0)=(q,p),\  (\xi,m\dot{\xi})(h)=(q',p')\right\}
\notag\\
&\stackrel{\eqref{def:Chbar}}={\abs{p'-p}^2}+12\left|\frac{m}{h}(q'-q)-\frac{p'+p}{2}\right|^2
+2h(p'-p)\cdot\nabla V(q)+h^2\abs{\nabla V(q)}^2
\notag\\
&=\abs{p'-p+h\nabla V(q)}^2+12\left|\frac{m}{h}(q'-q)-\frac{p'+p}{2}\right|^2.
\label{def:widehatCh}
\end{align}

Given a previous state $\rho_{k-1}$, define $\rho_k$ as the  solution of the minimization problem
\begin{equation}
\label{scheme:2b}
\min_\rho \;\frac1{2h}\frac1\gamma \widehat W_h(\rho_{k-1},\rho) + \mathcal A(\rho),
\end{equation}
where $\widehat W_h$ is the optimal-transport cost on $\R^{2d}$ with cost function~$\widehat C_h$.
\end{minipage}
}

Note how $\widehat C_h$  differs from~\eqref{def:widetildeCh} in that $\xi(t)$ is replaced by $q$ in $\nabla V$. This approximation is exact when $V$ is linear. We prove the convergence of solutions of Scheme 2b in Sections~\ref{section:EulerLagrangEquation}-\ref{section:Convergence}.

\bigskip

Neither of the costs $\widetilde C_h$ and $\widehat C_h$ gives rise to a metric, since they are asymmetric and do not vanish when $(q',p')=(q,p)$. It is possible to construct a two-step scheme with a symmetric cost and corresponding metric $W_h$.

\medskip\noindent
\framebox{
\begin{minipage}{\boxwidth}\noindent\textbf{Scheme 2c.}
Define
\begin{equation}
C_h(q,p;q',p')\colonequals|p'-p|^2+12\left|\frac{m}{h}(q'-q)-\frac{p'-p}{2}\right|^2+2m(q'-q)\cdot(\nabla V(q')-\nabla V(q)).\label{costfunct}
\end{equation}
Assume $\rho_{k-1}^h$ is given, define the single-step, backwards approximate streaming operator
\begin{equation}
\sigma_h(q,p)\colonequals\Bigl(q-h\frac{p}{m},p+h\nabla V(q)\Bigr).\label{Hamiltonmap}
\end{equation}

Given a previous state $\rho_{k-1}$, define $\rho_k$ in two steps.

\smallskip
\noindent\textbf{Hamiltonian step:} First determine $\mu_k^h(q,p)$ such that
 \begin{equation}
    \mu_k^h(q,p)\colonequals\sigma_h^{-1}(q,p)_{\sharp}\rho_{k-1}^h(q,p)\label{step1},
\end{equation}
where $\sharp$ denotes the push forward operator.

\smallskip
\noindent\textbf{Gradient flow step:}  Then determine $\rho^h_k$ that minimizes
\begin{equation}
 \min_\rho\; \frac{1}{2h}\frac1\gamma W_h(\mu_k^h,\rho)+\mathcal A(\rho)\label{step2},
\end{equation}
where $W_h$ is the metric on $\R^{2d}$ generated by the cost function~$C_h$.
\end{minipage}
}

\subsection{The main result and the relation to GENERIC}
\label{subsec:discussion}

The main theorem of this paper, Theorem~\ref{theo:maintheorem} below, states that the three new Schemes 2a-c are indeed approximation schemes for the Kramers equation~\eqref{KRequation}: the discrete-time approximate solutions constructed using each of these three schemes converge, as $h\to0$, to the unique solution of~\eqref{KRequation}.

This statement itself is a relatively uninteresting assertion: it states that the schemes are what we claim them to be, approximation schemes. The interest of this paper lies in the fact that these three schemes suggest a way towards a generalization of the theory of metric-space gradient flows, as developed in~\cite{AGS08}, to equations like~\eqref{KRequation} that combine dissipative with conservative effects.

Indeed, the full class of equations and systems that combines dissipative and conservative effects is extremely large. It contains the  Navier-Stokes-Fourier equations (which include heat generation and transport), systems modelling visco-elasto-plastic materials, relativistic hydrodynamics, many Boltzmann-type equations, and many other equations describing continuum-mechanical systems. In fact, the full class of systems covered by the GENERIC formalism~\cite{Oettinger05} is of this conservative-dissipative type, and indeed the Kramers equation is one of them.

The GENERIC class (General Equation for the Non-Equilibrium Reversible Irreversible Coupling) consists of equations for an unknown $x$ in a state space $\mathcal X$ that can be written as
\[
\dot x(t) = J(x) E'(x) + K(x) S'(x).
\]
Here $E,S:\mathcal X\to\R$ are functionals, and $J,K$ are operators. A GENERIC system is fully characterized by $\mathcal X$, $E$, $S$, $J$, and $K$. In addition, there are certain requirements on these elements, which include the \emph{symmetry conditions}
\[
\text{$J$\quad is antisymmetric\qquad and \qquad $K$\quad is symmetric and nonnegative},
\]
and the \emph{degeneracy} or \emph{non-interaction} conditions
\[
J(x)S'(x) = 0, \qquad K(x)E'(x) = 0, \qquad\text{for all }x\in \mathcal X.
\]
Because of these  properties, along a solution $E$ is constant and $S$ increases. In many systems the functionals $E$ and $S$ correspond to \emph{energy} and \emph{entropy}.

\medskip

When $F(p) = |p|^2/2m$, the Kramers equation~\eqref{KRequation} can be cast in this form.\footnote{In order to do this, the variable $\rho$ needs to be supplemented with an additional energy variable, that compensates for the gain and loss in the energy $\mathcal H$ as a result of the dissipative effects.} Because of this, the results of this paper strongly suggest that similar schemes can be constructed for arbitrary GENERIC systems. We leave this for future study.

\subsection{Conclusion and further discussion}

\noindent
We now make some further comments about the schemes of this paper.

\emph{Value of the three schemes.} Scheme 2a is in our opinion interesting because `it is the right thing to do'---it stays as close as possible to the underlying physics. However, its non-explicit nature makes it difficult to work with, as the calculations in the proof of Lemma~\ref{lemma:costproperties} illustrate. Scheme 2b is therefore useful as an approximation of Scheme 2a.
Scheme 2c has the advantage of being formulated in terms of a metric $W_h$, which suggests applicability of metric-space theory.

\emph{The linear-friction case $F(p)={|p|^2}/{2m}$}. The coefficient $\gamma kT$ in~\eqref{KRequation} and the coefficient $\sigma := \sqrt{2\gamma kT}$ in~\eqref{sto2} are obviously related by $\sigma^2=2\gamma kT$. When $F(p) = |p|^2/2m$, the coefficient $\gamma$ is also the coefficient of linear friction, and this relationship between $\sigma$, $\gamma$, and temperature is the one given by the fluctuation-dissipation theorem. This guarantees that the Boltzmann distribution
\begin{equation}
\label{def:rhoinfty}
\rho_\infty(q,p)=Z^{-1}\exp\left(-\frac{1}{kT}H(q,p)\right),
\end{equation}
is the unique stationary solution of \eqref{KRequation}. Moreover, the total free energy $\E$ is the relative entropy with respect to $\rho_\infty$, and it is a Lyapunov functional for the system, as is shown in~\eqref{eq:energydissipation}.

When $F$ does not have this specific form, but does have appropriate growth at infinity, then there still exists a unique stationary solution $\rho_\infty$, which however does not have the convenient characterization~\eqref{def:rhoinfty}. The relative entropy with respect to $\rho_\infty$ is then again a Lyapunov fucntional.

\emph{Connection to  ultra-parabolic equations.} If $V$ is linear, $V(q)=c\cdot q$, where $c\in \R^d$ is a constant vector,  then $\widehat{C}_h$ coincides with $\widetilde{C}_h$.
In this case, $\widehat{C}_h=\widetilde C_h$ is closely related to the fundamental solution of the  equation
\begin{equation}
\label{eq:untraparabolic}
\partial_t\rho(t,q,p)=-\frac{p}{m}\cdot\nabla_q\rho(t,q,p)+c\cdot\nabla_p\rho(t,q,p)+\frac{\sigma^2}{2}\Delta_p \rho(t,q,p).
\end{equation}
Indeed, the fundamental solution $\Gamma_t(q,p;q',p')$ of \eqref{eq:untraparabolic} is given by
\begin{equation}
\label{eq:fundamentalsolution}
\Gamma_t(q,p;q',p')=\frac{\alpha_1}{t^{2d}}\exp\left(-\frac{\gamma}{ \sigma^2t}\widehat{C}_t(q,p;q',p')\right),
\end{equation}
where $\alpha_1$ is a normalization constant depending only on $d$. This fact is true for a much more general linear system and is related to the \emph{controllable} property of the system~\cite{DelarueMenozzi10}. The appearance of the rate functional from the Freidlin-Wentzell theory in~\eqref{eq:fundamentalsolution} consolidates the connection to the large deviation principle of our aprroach.

\emph{Connection to the isentropic Euler equations.} The cost function $\overline{C}_h$ has been used in \cite{GW09,Westdickenberg10} to study the system of isentropic Euler equations,
\begin{align*}
&\partial_t\rho+\nabla\cdot(\rho\u)=0,
\\&\partial_t\u+\u\cdot\nabla\u=-\nabla U'(\rho),
\end{align*}
where $U\colon [0,\infty)\longrightarrow \R$ is an internal energy density. We now formally show the relationship between two equations. Suppose that $\rho(t,q,p)$ is a solution of the Kramers equation~\eqref{KRequation} with $F(p)=\abs{p}^2/2m$. We define the macroscopic spacial density and the bulk velocity as
\begin{align}
&\widetilde{\rho}(t,q)=\int_{\R^d}\rho(t,q,p)dp,
\\& \u(t,q)=\frac{1}{\widetilde{\rho}(t,q)}\int_{\R^d}\frac{p}{m}\rho(t,q,p)dp.
\end{align}
Using the so-called moment method, we find that $(\widetilde{\rho},\u)$ satisfies the following damped Euler equations~\cite{CSR96,Chavanis03,Chavanis04},
\begin{align}
&\partial_t\widetilde{\rho}+\nabla \cdot(\widetilde{\rho}\u)=0
\\&\partial_t\u+\u\cdot\nabla\u=-\frac{\kT}{m}\frac{\nabla \widetilde{\rho}}{\widetilde{\rho}}-\frac{1}{m} \nabla V-\frac{\gamma}{m}\u.
\end{align}
If $\gamma=0$ and $V\equiv0$, these are the isentropic Euler equations with  internal energy $U(\rho)=\kT\rho\log\rho$. In \cite{GW09,Westdickenberg10}, the authors showed that the isentropic Euler equations may be interpreted as a second-order differential equation in the space of probability measures. They introduced a discrete approximation scheme, which is similar to Schemes 2a-b, using the cost functional $\overline{C}_h$. One future topic of research is to analyse whether one can approximate other second-order differential equations in the space of probability measures (e.g., the Schr\"{o}dinger equation~\cite{VonRenesse11}), using the cost function $\widetilde{C}_h$.

\emph{Connection to Ambrosio-Gangbo~\cite{AG08}.} The Hamiltonian step in Scheme 2c is a generalization of the implicit Euler method for a finite-dimensional Hamiltonian system to an infinite-dimensional case. It is also compatible with the concept of Hamiltonian flows in the Wasserstein space of probability measures defined by Ambrosio and Gangbo in~\cite{AG08}. Let $\H\colon \P_2({\R^{2d}})\rightarrow (-\infty, +\infty]$ and $\overline{\mu}\in  \P_2({\R^{2d}})$ be given. Then $\mu_t\colon [0,\infty)\rightarrow  \P_2({\R^{2d}})$ is called a Hamiltonian flow of $\H$ with the initial measure $\overline{\mu}$ if the following equation holds
 \begin{equation*}
 \frac{d}{dt}\mu_t=\mathrm{div}_{qp}(\mu_tJ\nabla \H(\mu_t)), ~~ \mu_0= \overline{\mu}, ~~ t\in (0,T),
 \end{equation*}
 where $J$ is a skew-symmetric matrix and $\nabla \H(\mu_t)$ is the gradient of the Hamiltonian $\H$ at $\mu_t$ (Definition 3.2 in~\cite{AG08}). In particular, when $\H(\rho)=\int_{\R^{2d}}\left(\frac{p^2}{2m}+V(q)\right)\rho(q,p)dqdp$ then $\nabla \H =(\nabla_qV(q),\frac{p}{m})^T$.
 According to Lemma 6.2 in~\cite{AG08} when $\overline{\mu}$ is regular, a Hamiltonian flow in a small interval $(0,h)$ is constructed by pushing forward the initial measure $\overline{\mu}$ under the map $\Phi(t,\cdot)=(q(t),p(t))$ which is the solution of the system~\eqref{eq:SDE} (with $\gamma=0$). In the Hamiltonian step we approximate this system by the implicit Euler method and define $\mu_k^h$ to be the end point $\mu(h)$.


\subsection{Overview of the paper}

The paper is organized as follows. In Section~\ref{section:results}, we describe our assumptions and state the main result. Section~\ref{sec:costproperties} establishes some properties of the three cost functions. The proof of the main theorem is given  in Sections~\ref{section:EulerLagrangEquation} to \ref{section:Convergence}. In Section~\ref{section:EulerLagrangEquation}, we establish the Euler-Lagrange equations for the minimizers in three schemes. In Section~\ref{section:priorestimate}, we prove the boundedness of the second moments and the entropy functional. Finally, the convergence result is given in  Section~\ref{section:Convergence}.

\newpage

\section{Assumptions and main result}
\label{section:results}

Throughout the paper we make the following assumptions.
\begin{equation}
V\in C^3(\R^d)\text{ and }F\in C^2(\R^d),\  F(x)\geq 0 \text{ for all }x\in \R^d.\label{assumpt1}
\end{equation}
There exists a constant $C>0$ such that for all $z_1,z_2\in \R^d$
\begin{subequations}
\label{assumpt}
\begin{align}
&\frac1C\abs{z_1-z_2}^2\leq (z_1-z_2)\cdot (\nabla V(z_1)-\nabla V(z_2)),\label{assumpt2}
\\& \abs{\nabla V(z_1)-\nabla V(z_2)}\leq C\abs{z_1-z_2},\label{assumpt3}
\\& \abs{\nabla F(z_1)-\nabla F(z_2)}\leq C\abs{z_1-z_2},\label{assumpt4}\\
&\abs{\nabla^2 V(z_1)}, \abs{\nabla^3 V(z_1)} \leq C.\label{assumpt5}
\end{align}
\end{subequations}
Note that~\eqref{assumpt2} implies that $V$ increases quadratically at infinity, and therefore $V$ achieves its minimum. Without loss of generality we assume that this minimum is at the origin, which implies the estimate
\begin{equation}
\label{assumpt6}
|\nabla V(z)|\leq C|z|.
\end{equation}

As we remarked in the Introduction, we work in the dimensional setting, and keep all the physical constants in place, in order to make the physical background of the expressions clear. We make an important exception, however, for inequalities of the type above; here the constants $C$ can have any dimension, and we will group terms on the right-hand side of such estimates without taking their dimensions into account. This can be done without loss of generality, since we do not specify the generic constant $C$, and this constant will be allowed to vary from one expression to the next.

We only consider probability measures on $\R^{2d}$ which have a Lebesgue density, and we often tacitly  identify a probability measure with its density. We denote by $\P_2(\R^{2d})$ the set of all probability measures on $\R^d\times\R^d$ with finite second moment,
\begin{equation*}
\P_2(\R^{2d}):=\left\{\rho\colon \R^d\times\R^d\rightarrow[0,\infty) \text{ measurable}, \int_{\R^{2d}}\rho(q,p)dqdp=1, M_2(\rho)<\infty\right\},
\end{equation*}
where
\begin{equation}
\label{def:M2}
M_2(\rho)=\int_{\R^{2d}}(\gamma^2|q|^2+|p|^2) \rho(q,p)\,dqdp.
\end{equation}

With these assumptions, the functionals $\A$ and $\E$ introduced in the introduction are well-defined in $\P_2(\R^{2d})$. Moreover, the following two lemmas are now classical (see, e.g.,~\cite[Theorem 1.3]{Vil03},  \cite[Proposition 4.1]{JKO98}, and \cite[Lemma 4.2]{Hua00}). Let $C^*_h$ be one of $\widetilde C_h$, $\widehat{C}_h$, or $C_h$, defined in~\eqref{def:widetildeCh}, \eqref{def:widehatCh}, and~\eqref{costfunct}, with corresponding optimal-transport cost functional $W^*_h$.

\begin{lemma}\label{existoptimalplan}
Let $\rho_0,\rho\in \P_2(\R^{2d})$ be given. There exists a unique optimal plan $\PSopt\in \Gamma(\rho_0,\rho)$ such that
\begin{equation}
W^*_h(\rho_0,\rho)=\int_{\R^{4d}}C^*_h(q,p;q',p')\PSopt(dqdpdq'dp').
\end{equation}
\end{lemma}
\begin{lemma}\label{wellposedness}
Let $\rho_0\in\P_2(\R^{2d})$ be given. If $h$ is small enough, then the minimization problem
\begin{equation}
\min_{\rho\in\P_2(\R^{2d})}\frac{1}{2h}\frac1\gamma W^*_h(\rho_0,\rho)+\mathcal A(\rho),
\end{equation}
has a unique solution.
\end{lemma}
\noindent
These lemmas imply that Schemes 2a--c are well-defined.

Next, we make the definition of a weak solution precise.
A function $\rho\in L^1(\R^+\times\R^{2d})$ is called a weak solution of equation~\eqref{KRequation} with initial datum $\rho_0\in \P_2(\R^{2d})$ if it satisfies the  following weak formulation of (\ref{KRequation}):
\begin{multline}
 \label{weakKReqn}
\int_0^\infty\int_{\R^{2d}}\left[\partial_t \varphi + \frac{p}{m}\cdot\nabla_q\varphi-\bigl(\nabla_qV(q)+\gamma\nabla_pF(p)\bigr)\cdot \nabla_p\varphi+\gamma \kT\Delta_p\varphi\right]\rho\,dqdpdt
\\=-\int_{\R^{2d}}\varphi(0,q,p)\rho_0(q,p)\,dqdp,
\qquad\text{for all }\varphi\in  C_c^\infty(\R\times\R^{2d}).
\end{multline}
The main result of the paper is the following.

\begin{theorem}
\label{theo:maintheorem}
Let $\rho_0\in \P_2(\R^{2d})$ satisfy $\mathcal A(\rho_0) < \infty$. For any $h>0$ sufficiently small, let $\rho_k^h$ be the sequence of the solutions of any of the three Schemes 2a--c. For any $t\geq0$, define the piecewise-constant time interpolation
\begin{equation}
\rho^h(t,q,p)=\rho_k^h(q,p) \qquad \text{ for } (k-1)h<t\leq kh.\label{interpolation}
\end{equation}
Then for any $T>0$,
\begin{equation}
\rho^h\rightharpoonup \rho ~\text{ weakly in }~ L^1((0,T)\times\R^{2d})~\text{ as } h\to0,\label{weaklyconverence}
\end{equation}
where $\rho$ is the unique weak solution of the Kramers equation with initial value $\rho_0$. Moreover
\begin{equation}
\rho^h(t)\rightarrow\rho(t) ~ \text{ weakly in} ~ L^1(\R^{2d}) ~ \text{ as }~ h\to0 ~ \text{ for any }~ t>0, \label{pointwiseconvergence}
\end{equation}
and as $t\to0$,
\begin{equation}
\rho(t)\rightarrow\rho_0 ~\text{ in } ~ L^1(\R^{2d}). \label{intinialconvergence}
\end{equation}
\end{theorem}
\begin{proof}[Outline of the proof] The proof follows the procedure of \cite{JKO98} (see also \cite{Hua00,Hua11}) and is divided into three main steps, which are carried out in Sections~\ref{section:EulerLagrangEquation}, \ref{section:priorestimate}, and~\ref{section:Convergence}: establish the Euler-Lagrange equation for the minimizers, then estimate the second moments and entropy functionals, and finally pass to the limit $h\rightarrow 0$. We start in Section~\ref{sec:costproperties} with some properties of the cost functions.
\end{proof}

\section{Properties of the three cost functions}
\label{sec:costproperties}

Here we derive and summarize a number of properties of the three cost functions.
Define the  quadratic form
\[
N(q,p) := |\gamma q|^2 + |p|^2,
\]
so that $M_2(\rho) = \int_{\R^{2d}} N(q,p)\, \rho(q,p)\, dqdp$.

\begin{lemma}
\label{lemma:costproperties}
\begin{enumerate}
\item \label{lemma:costproperties:ineqs:qpC}
Let~$C_h^*$ be either $\widetilde C_h$ or $\widehat C_h$. There exists $C>0$ such that 
\begin{subequations}
\label{ineqs:qpC}
\begin{align}
|q-q'|^2+|p-p'|^2&\leq CC_h(q,p;q',p'),\label{ineq:qpC2c}\\
|q-q'|^2  &\leq Ch^2 \big[C_h^*(q,p;q',p') + N(q,p) + N(q',p')\big],\label{ineq:qC}\\
|p-p'|^2 &\leq C\big[C_h^*(q,p;q',p') + h^2 N(q,p) + h^2N(q',p')\big].\label{ineq:pC}
\end{align}
\end{subequations}
\item \label{lemma:costs:derivsTilde}
For the cost function $\widetilde C_h$ of Scheme~2a we have
\begin{subequations}
\label{derivs:tildeCh}
\begin{align}
\nabla_{q'}\widetilde{C}_h(q,p;q',p')&=\frac{24m}{ h}\left(\frac{m}{h}(q'-q)-\frac{p'+p}{2}\right) - 2h\nabla^2 V(q')\cdot p' + \sigma_h(q,p;q',p'),\label{derivs:tildeCh:q}
\\\nabla_{p'}\widetilde{C}_h(q,p;q',p')&=2(p'-p)-12\left(\frac{m}{h}(q'-q)-\frac{p'+p}{2}\right)+2h\nabla V(q') + \tau_h(q,p;q',p'),\label{derivs:tildeCh:p}
\end{align}
\end{subequations}
where there exists $C>0$ such that
\begin{multline}
|\sigma_h(q,p;q',p) |, \frac1h |\tau_h(q,p;q',p')|
\leq C h\Big\{\widetilde C_h(q,p;q',p') + N(q,p) + N(q',p')+1 \Big\}.
\label{est:sigmatau}
\end{multline}
\item \label{lemma:costs:derivsHat}
For the cost function $\widehat C_h$ of Scheme~2b we have
\begin{subequations}
\label{derivs:hatCh}
\begin{align}
\nabla_{q'}\widehat{C}_h(q,p;q',p')&=\frac{24m}{ h}\left(\frac{m}{h}(q'-q)-\frac{p'+p}{2}\right),\label{derivs:hatCh:q}
\\\nabla_{p'}\widehat{C}_h(q,p;q',p')&=2(p'-p)-12\left(\frac{m}{h}(q'-q)-\frac{p'+p}{2}\right)+2h\nabla V(q).\label{derivs:hatCh:p}
\end{align}
\end{subequations}
\item For the cost function $ C_h$ of Scheme~2c we have
\label{lemma:costs:derivsCh}
\begin{subequations}
\label{derivs:Ch}
\begin{align}
\nabla_{q'}C_h(q,p;q',p')&=\frac{24m}{h}\left(\frac{m}{h}(q'-q)-\frac{p'-p}{2}\right)+4m(\nabla V(q')-\nabla V(q))+r(q,q'),
\\\nabla_{p'}C_h(q,p;q',p')&=2(p'-p)-12\left(\frac{m}{h}(q'-q)-\frac{p'-p}{2}\right),
\end{align}
\end{subequations}
where
\begin{equation}
\label{est:r}
|r(q,q')|
\leq Ch^2 \big[C_h(q,p;q',p') + N(q,p) + N(q',p')\big].
\end{equation}
\end{enumerate}
\end{lemma}

\begin{proof}
For the length of this proof we fix $q,p,q',p'$, and $h$, and we abbreviate
\[
\overline C_h := \overline C_h(q,p;q',p'), \quad
\widetilde C_h := \widetilde C_h(q,p; q',p'),
\quad\text{and}\quad
N := N(q,p) + N(q',p') = |\gamma q|^2 + |p|^2 + |\gamma q'|^2 + |p'|^2.
\]
Let $\oxi(t)$ and $\txi(t)$, respectively, be the optimal curves in the definition of $\overline{C}_h$ in~\eqref{def:Chbar} and of $\widetilde{C}_h$ in~\eqref{def:widehatCh}. We will need a number of properties of these two curves. All the statements below are of the following type: there exists $C>0$ and $0<h_0<1$ such that the property holds for all $h<h_0$. Here $C$ is always independent of $q,p,q',p'$, and $h$. The norm $\|\cdot\|_p$ is the $L^p$-norm on the interval~$(0,h)$.

The curve $\oxi$ satisfies $\ddddot\oxi=0$, and hence it is a cubic polynomial
\begin{equation}\label{eq:simpleminimizer}
\overline{\xi}(t)=q_0+at+bt^2+ct^3,
\end{equation}
where the coefficients can be calculated from the boundary conditions:
\begin{equation*}
a=\frac{p}{m},\quad
b=\frac{3}{h^2}\left(q'-q-\frac{ph}{m}\right)-\frac{p'-p}{mh},
\quad c=\frac{p'+p}{mh^2}-\frac{2}{h^3}(q'-q).
\end{equation*}
Explicit calculations give
\begin{align}
\|\oxi\|_2^2 &\leq h \|\oxi\|_\infty^2 \leq ChN,\label{est:oxil2}\\
\|\ddot\oxi\|_2^2 &\leq h \|\ddot\oxi\|_\infty^2 \leq C\Big\{h^{-3} |q-q'|^2 + h^{-1}|p-p'|^2\Big\}, \label{est:oxippl2}\\
\|\ddot\oxi\|_1 &\leq h \|\ddot\oxi\|_\infty \leq C\Big\{h^{-1} |q-q'| + |p-p'|\Big\}.\label{est:oxippl1}
\end{align}

The curve $\txi(t)$ satisfies the equation
\begin{align}
\label{eq:ELeqnOptimalCurve}
&\mathcal N(\txi)(t) := m^2 \ddddot\txi(t)+2m\nabla^2V(\txi)\cdot {\ddot\txi}(t)+m\nabla^3V(\txi)\cdot {\dot\txi}\cdot {\dot\txi}(t)+\nabla^2V(\txi)\cdot \nabla V(\txi)(t)=0,\\
&(\txi,m{\dot\txi})(0)=(q,p),~(\txi,m{\dot\txi})(h)=(q',p'),\notag
\end{align}
where  $\nabla^3V$ is the third-order tensor of third derivatives of $V$. This is a relatively benign equation, but non-trivially nonlinear.

We will need the following four intermediate estimates:
\begin{align}
\|\txi\|_2^2 &\leq ChN,\label{est:tildexil2}\\
\overline C_h +  h\|\ddot\txi\,\|_2^2 &\leq C\big\{ \widetilde C_h + h^2 N\big\},\label{est:oChtCh}\\
\|\dot\txi\,\|_2^2 &\leq Ch\big\{ \widetilde C_h +  N\big\},\label{est:txip}\\
\|\ddddot u\.\|_1 & \leq C\big\{\widetilde C_h + N + 1\big\}.\label{est:upppp}
\end{align}
We first prove~\eqref{est:tildexil2}. Since $\txi$ is optimal in $\widetilde C_h$,
\begin{eqnarray}
m \|\ddot\txi\|_2 &\leq& \|m\ddot\txi + \nabla V(\txi)\|_2 + \|\nabla V(\txi)\|_2 \notag\\
&\stackrel{\eqref{def:widetildeCh}}\leq& \|m\ddot\oxi + \nabla V(\oxi)\|_2 + \|\nabla V(\txi)\|_2 \notag\\
&\leq& m \|\ddot\oxi\|_2 + \|\nabla V(\oxi)\|_2 + \|\nabla V(\txi)\|_2 \notag\\
&\stackrel{\eqref{assumpt6}}	\leq& m \|\ddot\oxi\|_2 + C\big(\|\oxi\|_2 + \|\txi\|_2\big)\notag\\
&\leq& m \|\ddot\oxi\|_2 + C\big(\|\oxi\|_2 + h^{1/2} \|\txi\|_\infty\big).
\label{est:txippl2}
\end{eqnarray}
Therefore
\begin{align*}
\|\txi\|_\infty &\leq  |\txi(0)| + h|\dot\txi(0)| +  h^{3/2} \|\ddot\txi\|_2\\
&\leq  |q| + \frac {h}m |p| +  Ch^{3/2}\Big\{\|\ddot\oxi\|_2  + \|\oxi\|_2 + h^{1/2} \|\txi\|_\infty\Big\}.
\end{align*}
If $h_0$ is small enough, then $Ch^2 < 1/2$, so that
\begin{align*}
\|\txi\|_\infty &\stackrel{\eqref{est:oxil2},\eqref{est:oxippl2}}\leq 2|q| + \frac {2h}m |p| + C\Big\{|q-q'| + h |p-p'|+  h^2\sqrt{N}\Big\}.
\end{align*}
Therefore
\[
\|\txi\|_2^2 \leq h\|\txi\|_\infty^2 \leq ChN,
\]
which is~\eqref{est:tildexil2}.

Similar to~\eqref{est:txippl2} it also follows, since $\txi$ is admissible for $\overline C_h$, that
\begin{align*}
\overline C_h &= m^2 h \|\ddot\oxi\|_2^2\leq m^2 h \|\ddot\txi\|_2^2  \leq 2h\|m\ddot\txi + \nabla V(\txi)\|_2^2 + 2h\|\nabla V(\txi)\|_2^2\\
&\stackrel{\eqref{def:widetildeCh},\eqref{assumpt6}}\leq 2\widetilde C_h + Ch\|\txi\|_2^2
  \stackrel{\eqref{est:tildexil2}}\leq 2\widetilde C_h + Ch^2 N,
\end{align*}
which implies~\eqref{est:oChtCh}.

We now can prove part~\ref{lemma:costproperties:ineqs:qpC} of the Lemma. \eqref{ineq:qpC2c} is a direct consequence of~\eqref{costfunct} and~\eqref{assumpt2}. The estimate for $p$ follows from~\eqref{def:widehatCh} and~\eqref{assumpt6} for $\widehat{C}_h$, and from~\eqref{def:Chbar} and~\eqref{est:oChtCh} for $\widetilde{C}_h$:
\begin{align*}
\abs{p'-p}^2&\leq C\Big[\abs{p'-p+h\nabla V(q)}^2+h^2\abs{\nabla V(q)}^2\Big]\leq C\Big[\widehat{C}_h(q,p;q',p')+h^2N\Big],\\
\abs{p'-p}^2&\leq \overline{C}_h\leq C(\widetilde{C}_h+h^2 N).
\end{align*}
Similarly,
\begin{align}
\abs{q'-q}^2&=\frac{h^2}{m^2}\left|\frac{m}{h}(q'-q)-\frac{p+p'}{2}+\frac{p'+p}{2}\right|^2\notag
\\&\leq \frac{3h^2}{m^2}\left(\left|\frac{m}{h}(q'-q)-\frac{p+p'}{2}\right|^2+\frac{|p|^2}{4}+\frac{|p'|^2}{4}\right)\notag
\\&\leq Ch^2(\overline{C}_h+N)\leq Ch^2(\widetilde{C}_h + N ),
\label{est:qChproof}
\end{align}
and also
\[
\abs{q'-q}^2\leq Ch^2(\widehat{C}_h + N).
\]

Using the Poincar\'e inequality $\|v-\dashint v\|_2 \leq Ch \|v'\|_2$, the estimate~\eqref{est:txip} then follows by
\[
\|\dot\txi\|^2_2 \leq 2\|{\textstyle\dashint \dot\txi}\|^2_2 + Ch^2\|\ddot\txi\|_2^2
\stackrel{\eqref{est:oChtCh}}\leq \frac2{h} |q-q'|^2 +  Ch\big\{ \widetilde C_h + h^2 N\big\}
\stackrel{\eqref{ineq:qC}}\leq Ch\big\{ \widetilde C_h  + N\big\}.
\]

To prove the final of the four intermediate estimates, \eqref{est:upppp}, we define $u= \txi-\oxi$; remark that
\begin{equation}
\label{eq:u}
m^2 \ddddot u = -2m\nabla^2V(\txi)\cdot \ddot\txi-m\nabla^3V(\txi)\cdot \dot\txi\cdot \dot\txi-\nabla^2V(\txi)\cdot \nabla V(\txi).
\end{equation}
Note that $u=\dot u=0$ at $t=0,h$, so that we have $\|u\|_1 \leq Ch^4\|\ddddot u\.\|_1$ and $\|\ddot u\|_1\leq C h^2 \|\ddddot u\.\|_1$. We then calculate
\begin{eqnarray*}
\|\ddddot u\.\|_1 &\stackrel{\eqref{eq:u},\eqref{assumpt}}\leq& C\Big\{\|\ddot\txi\|_1 + \|\dot\txi\|_2^2 + \|\txi\|_1 + h\Big\}\\
&\leq& C\Big\{\|\ddot\oxi\|_1 + \|\dot\txi\|_2^2 + \|\oxi\|_1 +  \|\ddot u\|_1 + \|u\|_1\Big\}\\
&\leq& C\Big\{\|\ddot\oxi\|_1 + \|\dot\txi\|_2^2 + \|\oxi\|_1 +  h^2\|\ddddot u\.\|_1 + h^4\|\ddddot u\.\|_1\Big\}.
\end{eqnarray*}
Again, taking $h_0$ sufficiently small, we have $C(h^2+h^4)<1/2$, and therefore
\begin{eqnarray*}
\|\ddddot u\.\|_1 &\leq& C\Big\{\|\ddot\oxi\|_1 + \|\dot\txi\|_2^2 + \|\oxi\|_1 \Big\}\\
&\stackrel{\eqref{est:oxil2},\eqref{est:oxippl1},\eqref{est:txip}}\leq& C\Big\{\frac{|q-q'|}h + |p-p'| +  h\widetilde C_h + h N + h\sqrt N \Big\}\\
&\stackrel{\eqref{ineq:qC}}\leq & C\Big\{\sqrt{\widetilde C_h + N} +  h\widetilde C_h  + N  + 1\Big\}\\
&\leq & C\Big\{\widetilde C_h  + N  + 1\Big\}.
\end{eqnarray*}

We now continue with parts~\ref{lemma:costs:derivsTilde}, \ref{lemma:costs:derivsHat}, and~\ref{lemma:costs:derivsCh}. The derivatives of $\widehat{C}_h$ can be calculated directly using the explicit expression~\eqref{def:widehatCh}.
The  derivatives of $\widetilde{C}_h$ can be calculated as follows. Let $\eta\in C^2([0,h];\R^{2d})$ satisfy $\eta(0) = 0$. Then
\begin{align*}
\lim_{\e\to0}\; &4\gamma \kT h\,\widetilde I(\txi+\e\eta) =
2h\int_0^h \big{(}m{\ddot\txi} + \nabla V(\txi)\big{)}\cdot \big{(}m{\ddot\eta} + \nabla^2 V(\txi)\cdot \eta\big{)}(t)\, dt\\
&= 2h\int_0^h \mathcal N(\txi)\cdot\eta(t)\, dt + 2h\Big[m\dot\eta\big{(}m{\ddot\txi} + \nabla V(\txi)\big{)} - m\eta\big{(}m{\dddot\txi} + \nabla^2 V(\txi)\cdot {\dot\txi}\big{)}\Big](h).
\end{align*}
Note that $\mathcal N(\txi)\equiv0$ by the stationarity~\eqref{eq:ELeqnOptimalCurve} of $\txi$. This expression is equal to
\[
\nabla_{q'}\widetilde{C}_h(q,p;q',p')\cdot \eta(h) + \nabla_{p'}\widetilde{C}_h(q,p;q',p')\cdot m\dot \eta(h),
\]
which allows us to identify the two derivatives in terms of $\txi$. Setting $u=\txi-\overline{\xi}$, we rewrite these in terms of $u$:
\begin{eqnarray*}
\nabla_{q'}\widetilde{C}_h(q,p;q',p')
&=&-2hm^2\dddot\txi(h)-2hm\nabla^2V(\txi(h))\cdot{\dot\txi}(h)
\\&=&-2hm^2\overline{\dddot\xi}(h)-2hm\nabla^2V(\txi(h))\cdot{\dot\txi}(h)-2hm^2\big{(}\dddot\txi(h)-\dddot\oxi(h)\big{)}
\\&\stackrel{\eqref{eq:simpleminimizer}}=&\frac{24m}{h}\left(\frac{m}{h}(q'-q)-\frac{p'+p}{2}\right)-2h\nabla^2 V(q')\cdot p'-2hm^2\dddot u(h),\\
%
\nabla_{p'}\widetilde{C}_h(q,p;q',p')
&=&2hm{\ddot\txi}(h)+2h\nabla V(\txi(h))
\\&=&2hm{\ddot\oxi}(h)+2h\nabla V(\txi(h))+2hm\big{(}{\ddot\txi}(h)-{\ddot\oxi}(h)\big{)}
\\&\stackrel{\eqref{eq:simpleminimizer}}=&2(p'-p)-12\left(\frac{m}{h}(q'-q)-\frac{p'+p}{2}\right)+2h\nabla V(q')+2hm\ddot u(h).
\end{eqnarray*}
Therefore~\eqref{derivs:tildeCh} holds with
\[
\sigma_h = -2hm^2 \dddot u(h) \qquad\text{and}\qquad \tau_h = 2hm\ddot u(h).
\]
The estimates~\eqref{est:sigmatau}  then follow from~\eqref{est:upppp} and the inequalities
\[
\|\ddot u\|_\infty \leq h\|\dddot u\.\|_\infty \leq Ch\|\ddddot u\.\|_1,
\]
which hold since $u=\dot u=0$ at $t=0,h$.

The derivatives of $C_h$ are given by~\eqref{derivs:Ch}, where
\[
r(q,q') := 2m\Big[\nabla^2 V(q')\cdot (q'-q) - \nabla V(q') + \nabla V(q)\Big].
\]
The estimate~\eqref{est:r} on $r$ follows from~\eqref{assumpt5}, \eqref{est:qChproof}, and the fact that by~\eqref{assumpt2}, $\overline C_h\leq C_h$.
\end{proof}

\section{The Euler-Lagrange equation for the minimization problem}
\label{section:EulerLagrangEquation}
Let $C^*_h$ be one of $\widetilde C_h$, $\widehat{C}_h$, or $C_h$, defined in~\eqref{def:widetildeCh}, \eqref{def:widehatCh}, and~\eqref{costfunct}, with corresponding optimal-transport cost functional $W^*_h$.
Let $\overline\rho\in\P_2(\R^{2d})$ be given and let $\rho$ be the unique solution of the minimization problem
\begin{equation*}
\min_{\mu\in\P_2(\R^{2d})}\frac{1}{2\gamma h}W^*_h(\overline\rho,\mu)+\mathcal A(\mu).
\end{equation*}

We now establish the Euler-Lagrange equation for $\rho$. Following the now well-established route (see e.g.~\cite{JKO98,Hua00}), we first define a perturbation of $\rho$ by a push-forward under an appropriate flow.
Let $\xi,\eta\in C_0^\infty(\R^{2d},\R^d)$. We define the flows
$\Phi,\Psi\colon[0,\infty)\times\R^{2d}\rightarrow\R^d$ such that
\begin{eqnarray*}
&&\frac{\partial\Psi_s}{\partial s}=\phi(\Psi_s,\Phi_s),~ \frac{\partial\Phi_s}{\partial s}=\eta(\Psi_s,\Phi_s), \\
&& \Psi_0(q,p)=q,~\Phi_0(q,p)=p.
\end{eqnarray*}
Let $\rho_s(q,p)$ be the push forward of $\rho(q,p)$ under the flow $(\Psi_s,\Phi_s)$, i.e., for any $\varphi\in C_0^\infty(\R^{2d},\R)$ we have
\begin{equation}
\int_{\R^{2d}}\varphi(q,p)\rho_s(q,p)dqdp=\int_{\R^{2d}}\varphi(\Psi_s(q,p),\Phi_s(q,p))\rho(q,p)dqdp. \label{pushforward}
\end{equation}
Obviously $\rho_0(q,p)=\rho(q,p)$, and an explicit calculation gives
\begin{equation}
\partial_s\rho_s\big|_{s=0} =-\text{div}_q\rho\phi-\text{div}_p\rho\eta \qquad \text{in the sense of distributions}.
\end{equation}
By following the calculations in e.g.~\cite{Hua00} we then compute the stationarity condition on $\rho$,
\begin{align}
0&=\frac{1}{2\gamma h}\int_{\R^{4d}}\left[\nabla_{q'}C^*_h(q,p;q',p')\cdot\phi(q',p')+\nabla_{p'}C^*_h(q,p;q',p')\cdot\eta(q',p')\right]\PSopt(dqdpdq'dp')\nonumber
\\&\qquad+\int_{\R^{2d}}\rho(q,p)\nabla_pF(p)\cdot\eta(q,p)dqdp-\kT\int_{\R^{2d}}\rho(q,p)\left[\text{div}_q\phi(q,p)+\text{div}_p\eta(q,p)\right]dqdp,
\label{EuLageqn}
\end{align}
where $\PSopt$ is optimal in $W_h^*(\overline\rho,\rho)$.
For any $\varphi\in C_0^\infty(\R^{2d},\R)$, we choose
\begin{align*}
\phi(q',p')&=-\frac{\gamma h^2}{6m^2}\nabla_{q'}\varphi(q',p')+\frac{\gamma h}{2m}\nabla_{p'}\varphi(q',p'),\nonumber
\\\eta(q',p')&=-\frac{\gamma h}{2m}\nabla_{q'}\varphi(q',p')+\gamma \nabla_{p'}\varphi(q',p').
\end{align*}
i.e.,
\begin{equation}\label{xiphi}
\begin{pmatrix}
\phi\\
\eta
\end{pmatrix}=
\begin{pmatrix}
-\frac{\gamma h^2}{6m^2}I&\frac{\gamma h}{2m}I\\
-\frac{\gamma h}{2m}I&\gamma I
\end{pmatrix}\nabla\varphi(q',p').
\end{equation}
Now the specific form of the cost functional $C^*_h(q,p;q',p')$ comes into play. We calculate the gradient expression in \eqref{EuLageqn} for each scheme in the next subsections.

\begin{remark}
The structure of the choice~\eqref{xiphi} can be understood in terms of the conservative-dissipative nature of the Kramers equation.
The matrix in front of $\nabla \varphi(q',p')$ in \eqref{xiphi} is of the form
\begin{equation*}
\begin{pmatrix}
-\frac{\gamma h^2}{6m^2}I&\frac{\gamma h}{2m}I\\
-\frac{\gamma h}{2m}I&\gamma I
\end{pmatrix}=
\underbrace{\begin{pmatrix}
-\frac{\gamma h^2}{6m^2}I&0\\
0&\gamma I
\end{pmatrix}}_{A}-\underbrace{\frac{\gamma h}{2m}\begin{pmatrix}
0&I\\
-I&0
\end{pmatrix}}_{B}.
\end{equation*}
Note that $A$ is symmetric and $B$ is antisymmetric: this mirrors the conservative-dissipative structure of the Kramers equation.

The top-left block in $A$, which would correspond to diffusion in the spatial variable $q$, is of order $O(h^2)$, and therefore vanishes when $h\to0$. The other block, which corresponds to  diffusion in the momentum variable $p$, is of order $O(1)$ and remains. This explains how in the limit $h\rightarrow 0$ only diffusion in the momentum variable remains.
\end{remark}

\subsection{Schemes 2a and 2b}

\begin{lemma}
\label{lemma:realcostEL}
Let $h>0$ and let $\{\rho^h_k\}$ be the sequence of the minimizers  either for problem~\eqref{scheme:2a} in Scheme~2a or for problem~\eqref{scheme:2b} in Scheme~2b. Let $W_h^*$ be $\widetilde W_h$ for Scheme~2a and $\widehat W_h$ for Scheme~2b, and let $P_k^{h*}$ be optimal in $W_h^*(\rho_{k-1}^h,\rho_k^h)$. Then, for all $\varphi\in C_c^\infty(\R^{2d})$, there holds
\begin{align}
\label{eq:realcostEL}
0&=\frac{1}{h}\int_{\R^{4d}}\left[(q'-q)\cdot\nabla_{q'}\varphi(q',p')+(p'-p)\cdot\nabla_{p'}\varphi(q',p')\right]{P}_k^{h*}(dqdpdq'dp')\nonumber
\\&\quad -\frac{1}{m}\int_{\R^{2d}}p'\cdot\nabla_{q'}\varphi(q',p')\rho_k^h(q',p')dq'dp'+\int_{\R^{2d}}\nabla V(q')\cdot \nabla_{p'}\varphi(q',p')\rho_k^h(q',p')dq'dp'\nonumber
\\&\quad +\gamma\int_{\R^{2d}}\left[\nabla F(p')\cdot\nabla_{p'}\varphi(q',p')-\kT\Delta_{p'}\varphi(q',p')\right]\rho_k^h(q',p')dq'dp'+\omega_k^h,
\end{align}
where
\[
|\omega_k^h|\leq Ch\Big[  W_h^*(\rho_{k-1}^h,\rho_k^h) + M_2(\rho_{k-1}^h) + M_2(\rho_k^h) + 1\Big].
\]
\end{lemma}
\noindent
The second moment $M_2$ is defined in~\eqref{def:M2}.

\begin{proof}
For Scheme 2b we combine~\eqref{xiphi} with~\eqref{derivs:hatCh} to yield
\begin{align}
&\nabla_{q'}\widehat{C}_h(q,p;q',p')\cdot\phi(q',p')+\nabla_{p'}\widehat{C}_h(q,p;q',p')\cdot\eta(q',p')\nonumber
\\&\qquad=2\gamma\left[(q'-q)\cdot\nabla_{q'}\varphi(q',p')+(p'-p)\cdot\nabla_{p'}\varphi(q',p')-\frac{h}{m}p'\cdot\nabla_{q'}\varphi(q',p')\right]\nonumber
\\&\qquad\qquad+2\gamma\nabla V(q)\cdot\left[-\frac{h^2}{2m}\nabla_{q'}\varphi(q',p')+h\nabla_{p'}\varphi(q',p')\right].\label{1stepGradC}
\end{align}
Substituting~\eqref{xiphi} and~\eqref{1stepGradC} into the Euler-Lagrange equation~\eqref{EuLageqn}, we obtain
\begin{align}
0&=\frac{1}{h}\int_{\R^{4d}}\left[(q'-q)\cdot\nabla_{q'}\varphi(q',p')+(p'-p)\cdot\nabla_{p'}\varphi(q',p')\right]\widehat P_k^h(dqdpdq'dp')\nonumber
\\&\quad-\frac{1}{m}\int_{\R^{2d}}p'\cdot\nabla_{q'}\varphi(q',p')\rho_k^h(q',p')dq'dp'+\int_{\R^{4d}}\nabla V(q)\cdot \nabla_{p'}\varphi(q',p')\widehat P_k^h(dqdpdq'dp')\nonumber
\\&\quad+\gamma\int_{\R^{2d}}\left[\nabla F(p')\cdot\nabla_{p'}\varphi(q',p')+ \kT\frac{h^2}{6m^2} \Delta_{q'} \varphi(q',p')-\kT \Delta_{p'}\varphi(q',p')\right]\rho_k^h(q',p')dq'dp'\nonumber
\\&\quad-\frac{h}{2m}\int_{\R^{4d}}\big[\nabla V(q)+\gamma \nabla F(p')\big]\cdot\nabla_{q'}\varphi(q',p')\widehat P_k^{h}(dqdpdq'dp')
\label{1stepELeqn}.
\end{align}
Therefore~\eqref{eq:realcostEL} holds with
\begin{eqnarray*}
|\omega_k^h| &=& \bigg|\int_{\R^{4d}}(\nabla V(q)-\nabla V(q'))\cdot \nabla_{p'}\varphi(q',p')\widehat P_k^{h}(dqdpdq'dp')dq'dp'
\\
&&\quad + \kT\frac{h^2}{6m^2}\int_{\R^{2d}} \Delta_{q'} \varphi(q',p')\rho_k^h(q',p')dq'dp'\\
&&\quad-\frac{h}{2m}\int_{\R^{4d}}\big[\nabla V(q)+\gamma \nabla F(p')\big]\cdot\nabla_{q'}\varphi(q',p')\widehat P_k^{h}(dqdpdq'dp')\bigg|\\
&\stackrel{\eqref{assumpt3},\eqref{assumpt4}}\leq& C\int_{\R^{4d}} \big[ |q-q'| + h(|q| + |p'| + 1)\big]\widehat P_k^{h}(dqdpdq'dp')\\
&\leq &C\int_{\R^{4d}} \Big[ \frac1h |q-q'|^2 + h(|q|^2 + |p'|^2 + 1)\Big]\widehat P_k^{h}(dqdpdq'dp')\\
&\stackrel{\eqref{ineqs:qpC}}\leq & Ch\Big[ \widehat W_h(\rho_{k-1}^h,\rho_k^h) + M_2(\rho_{k-1}^h) + M_2(\rho_k^h) + 1\Big].
\end{eqnarray*}
This proves Lemma~\ref{lemma:realcostEL} for Scheme~2b.

\medskip
For Scheme 2a we obtain an identity similar to~\eqref{1stepGradC},
\begin{align}
&\nabla_{q'}\widetilde{C}_h(q,p;q',p')\cdot\phi(q',p')+\nabla_{p'}\widetilde{C}_h(q,p;q',p')\cdot\eta(q',p')\nonumber
\\&\qquad=2\gamma\left[(q'-q)\cdot\nabla_{q'}\varphi(q',p')+(p'-p)\cdot\nabla_{p'}\varphi(q',p')-\frac{h}{m}p'\cdot\nabla_{q'}\varphi(q',p')\right]\nonumber
\\&\qquad\qquad+2\gamma\Big\{h\nabla V(q') + \frac12 \tau_h(q,p;q',p')\Big\}\cdot\left[-\frac{h}{2m}\nabla_{q'}\varphi(q',p')+\nabla_{p'}\varphi(q',p')\right] \notag\\
&\qquad\qquad+ 2\gamma\Bigl\{-h\nabla ^2V(q')\cdot p' + \frac12 \sigma_h(q,p';q',p')\Big\}\cdot \left[-\frac{h^2}{6m^2} \nabla_{q'}\varphi(q',p')+\frac h{2m}\nabla_{p'}\varphi(q',p')  \right].\notag
\end{align}
This leads to the same equation as~\eqref{eq:realcostEL}, but now with error term
\begin{align*}
\omega_k^h&=-\frac{h}{2m}\int_{\R^{4d}}\nabla V(q')\cdot\nabla_{q'}\varphi(q',p')\widetilde{P}_k^h(dqdpdq'dp')\nonumber\\
&\quad+\int_{\R^{4d}}\Big\{\nabla ^2V(q')\cdot p' - \frac1{2h} \sigma_h(q,p;q',p')\Big\}\cdot\left[\frac{h^2}{6m^2}\nabla_{q'}\varphi(q',p')-\frac{h}{2m}\nabla_{p'}\varphi(q',p')\right]\widetilde{P}_k^h(dqdpdq'dp')\notag\\
&\quad + \frac1{2h} \int_{\R^{4d}} \tau_h(q,p;q',p') \left[-\frac{h}{2m}\nabla_{q'}\varphi(q',p')+\nabla_{p'}\varphi(q',p')\right]\widetilde{P}_k^h(dqdpdq'dp') \notag\\
&\quad-\frac{\gamma h}{2m}\int_{\R^{4d}}\nabla F(p')\cdot\nabla_{q'}\varphi(q',p')\rho_k^h(q',p')dq'dp'\nonumber\\
&\quad + \kT\frac{h^2}{6m^2}\int_{\R^{2d}} \Delta_{q'} \varphi(q',p')\rho_k^h(q',p')dq'dp'.\notag
\end{align*}
We estimate this error as follows, using the notation of the proof of Lemma~\ref{lemma:costproperties:ineqs:qpC}:
\begin{eqnarray*}
|\omega_k^h| &\leq & C\int_{\R^{4d}} \biggl\{ h(1+|q'|) + h|p'| + |\sigma_h| + \frac1h |\tau_h| + h(1+|p'|) + h^2\biggr\} \widetilde{P}_k^h\\
&\leq &C \int_{\R^{4d}} \biggl\{ h(1 + |q'|^2 + |p'|^2) +h\big[\widetilde C_h + N + 1\bigr]\biggr\}\widetilde{P}_k^h\\
&\leq &Ch \int_{\R^{4d}} \bigl[\widetilde C_h + N + 1\bigr] \widetilde{P}_k^h\\
&\leq & Ch\Big[ \widetilde W_h(\rho_{k-1}^h,\rho_k^k) + M_2(\rho_{k-1}^h) + M_2(\rho_k^h) +1 \Big].
\end{eqnarray*}
This concludes the proof of Lemma~\ref{lemma:realcostEL}.
\end{proof}

\subsection{Scheme 2c}

\begin{lemma}
\label{firstvariationLemma}
Let $h>0$ and let $\{\mu_k^h\}$ and $\{\rho^h_k\}$ be the sequences constructed in Scheme 2c. Let $P_k^h(dqdpdq'dp')$ be the optimal plan in the definition of $W_h(\mu_k^h,\rho_k^h)$. Then, for all $\varphi\in C_c^\infty(\R^{2d})$, there holds
\begin{align}
0&=\frac{1}{h}\int_{\R^{4d}}\left[(q'-q+\frac{p}{m}h)\cdot\nabla_{q'}\varphi(q',p')+(p'-p-h\nabla_qV(q))\cdot\nabla_{p'}\varphi(q',p')\right]P_k^h(dqdpdq'dp')\nonumber
\\&\quad-\frac{1}{m}\int_{\R^{2d}}p\cdot\nabla_q\varphi(q,p)\rho_k^h(dqdp)+\int_{\R^{2d}}\nabla V(q)\cdot \nabla_p\varphi(q,p)\rho_k^h(q,p)dqdp\nonumber
\\&\quad +\gamma\int_{\R^{2d}}\left[\nabla F(p)\cdot\nabla_p\varphi(q,p)-\kT \Delta_p\varphi(q,p)\right]\rho_k^h(q,p)dqdp+\zeta_k^h,
\label{EulerLagrangetemp}
\end{align}
where
\begin{align*}
|\zeta_k^h|\leq Ch\big[ h W_h(\mu_k^h, \rho_k^h) + M_2(\mu_k^h) + M_2(\rho_k^h) + 1].
\end{align*}
\end{lemma}

\begin{proof}
From~\eqref{xiphi} and~\eqref{derivs:Ch} we obtain
\begin{align}
&\nabla_{q'}C_h(q,p;q',p')\cdot\phi(q',p')+\nabla_{p'}C_h(q,p;q',p')\cdot\eta(q',p')\nonumber
\\&\qquad=2\gamma\left[(q'-q)\cdot\nabla_{q'}\varphi(q',p')+(p'-p)\cdot\nabla_{p'}\varphi(q',p')-\frac{h}{m}(p'-p)\cdot\nabla_{q'}\varphi(q',p')\right]\nonumber
\\&\qquad\qquad+\gamma\Big[4m(\nabla V(q')-\nabla V(q)) + r(q,q')\Big] \cdot\left\{-\frac{h^2}{6m^2}\nabla_{q'}\varphi(q',p')+ \frac h{2m} \nabla_{p'}\varphi(q',p')\right\}.
\label{GradC}
\end{align}
Substituting~\eqref{xiphi} and~\eqref{GradC} into the Euler-Lagrange equation~\eqref{EuLageqn}, we obtain
\begin{align}
0&=\frac{1}{h}\int_{\R^{4d}}\left[(q'-q)\cdot\nabla_{q'}\varphi(q',p')+(p'-p)\cdot\nabla_{p'}\varphi(q',p')\right]P_k^h(dqdpdq'dp')\nonumber
\\&\quad-\frac{1}{m}\int_{\R^{4d}}(p'-p)\cdot\nabla_{q'}\varphi(q',p')P_k^h(dqdpdq'dp')+\int_{\R^{4d}}(\nabla V(q')-\nabla V(q))\cdot \nabla_{p'}\varphi(q',p')P_k^h(dqdpdq'dp')\nonumber
\\&\quad +\gamma\int_{\R^{2d}}\left[\nabla F(p)\cdot\nabla_p\varphi(q,p)-\kT \Delta_p\varphi(q,p)\right]\rho_k^h(q,p)dqdp
+\zeta_k^h\label{ELeqn},
\end{align}
where we estimate the remainder, again using the notation of the proof of Lemma~\ref{lemma:costproperties},
\begin{eqnarray*}
|\zeta_k^h|&=&\Bigg|-\frac{h}{3m}\int_{\R^{4d}}(\nabla V(q')-\nabla V(q))\cdot \nabla_{q'}\varphi(q',p')P_k^h(dqdpdq'dp')
\\&&\qquad+\frac12 \int_{\R^{4d}}r(q,q')\cdot\left[-\frac{h}{6m^2}\nabla_{q'}\varphi(q',p')+\frac1{2m}\nabla_{p'}\varphi(q',p')\right] P_k^h(dqdpdq'dp')
\\&&\qquad-\frac{\gamma h}{2m}\int_{\R^{2d}}\rho_k^h(q,p)\nabla F(p)\cdot \nabla_q\varphi(q,p)dqdp+\kT \frac{\gamma h^2}{6m^2}\int_{\R^{2d}}\rho_k^h(q,p)\Delta_q \varphi(q,p)dqdp\Bigg|\\
&\stackrel{\eqref{assumpt},\eqref{est:r}}\leq& C\int_{\R^{4d}} \big[ h |q'-q| + h^2(C_h + N) + h(1+|p'|) + h^2\big]\, P_k^h(dqdpdq'dp')
\\
&\leq& C\int_{\R^{4d}} \big[ h (|q|^2 + |q'|^2) + h^2(C_h + N) + h(1+|p'|^2) \big]\, P_k^h(dqdpdq'dp')\\
&\leq & Ch\big[ h W_h(\mu_k^h, \rho_k^h) + M_2(\mu_k^h) + M_2(\rho_k^h) + 1].
\end{eqnarray*}
This concludes the proof of Lemma~\ref{firstvariationLemma}.
\end{proof}

\section{A priori estimate: Boundedness of the second moment and entropy }
\label{section:priorestimate}
This section includes some technical lemmas that are needed in order to prove the convergence result of Section~\ref{section:Convergence}.

\begin{lemma}\label{priorbound1:schemes2ab}
Let $\{\rho_k^h\}_{k\geq1}$ be the sequence of the minimizers of Scheme 2a or Scheme 2b for fixed $h>0$.  Then for any positive integer $n$ and sufficiently small $h$, we have
\begin{equation}
\sum_{k=1}^n{W}^*_h(\rho_{k-1}^h,\rho_k^h)\leq
2\gamma h(\A(\rho_0)-\A(\rho_n^h))+Ch^2\sum_{k=0}^{n}M_2(\rho_k^h)+Cnh^2,\label{sumIneqn:scheme2a}
\end{equation}
for some constant $C>0$ independent of $n$, where $W^*_h$ is either $\widetilde W_h$ or $\widehat W_h$. Similarly, if $\{\mu_k^h\}$ and $\{\rho_k^h\}$ are the sequences constructed in Scheme 2c, then
\begin{equation*}
\sum_{k=1}^n{W}_h(\mu_k^h,\rho_k^h)\leq
2\gamma h(\A(\rho_0)-\A(\rho_n^h))+Ch^2\sum_{k=0}^{n}M_2(\rho_k^h)+Cnh^2.
\end{equation*}
\end{lemma}
\begin{proof}
We give the details for Scheme 2a and then comment on the differences for the other schemes.
We first define the operator $\mathbf s_h:\R^{2d}\to\R^{2d}$ as the solution operator over time $h$ for the Hamiltonian system
\begin{equation}
\label{eq:HamSys}
 Q'  = \frac Pm, \quad P' = -\nabla V(Q),
\end{equation}
that is, $\mathbf s_h(q,p)$ is the solution at time $h$ given the initial datum $(q,p)$ at time zero. The operator $\mathbf s_h$ is bijective and volume-preserving.

For any fixed $k\geq 1$, $\rho_k^h$ minimizes the functional $(2h\gamma)^{-1}\widetilde{W}_h(\rho_{k-1}^h,\rho)+\A(\rho)$ over $\rho\in \P_2(\R^{2d})$, i.e.,
\begin{equation}
\widetilde{W}_h(\rho_{k-1}^h,\rho_k^h)+2h\gamma\A(\rho_k^h)\leq \widetilde{W}_h(\rho_{k-1}^h,\rho)+2h\gamma\A(\rho),
\end{equation}
for every $\rho\in \P_2(\R^{2d})$.  In particular by taking $\rho=(\textbf{s}_h^{-1})_\sharp\rho_{k-1}^h=:\rho_*^h$, for which $\widetilde{W}_h(\rho_{k-1}^h,\rho_*^h)=0$, it follows that
\begin{equation}
\widetilde{W}_h(\rho_{k-1}^h,\rho_k^h)\leq 2\gamma h\big[\A(\rho_*^h)-\A(\rho_k^h)\big]
=2\gamma h\big[\F(\rho_*^h)-\F(\rho_k^h)\big]+2\gamma h\big[S(\rho_*^h)-S(\rho_k^h)\big]. \label{inequality:schemes2ab}
\end{equation}
We now estimate each term on the right hand side. Write $(\overline{q},\overline{p})=\mathbf{s}_h(q,p)$. Using  equation~\eqref{eq:HamSys}, we readily estimate that the solution $(Q(t),P(t))$ starting at $(q,p)$ and ending at $(\overline q,\overline p)$ satisfies $\|Q\|_\infty\leq C\left(|\overline{q}|+h|\overline{p}|\right)$, and therefore
\begin{equation*}
\left|\int_0^h\nabla V(Q(t))dt\right|\leq h\sup_{t\in[0,h]}|\nabla V(Q(t))|\leq h\|Q\|_\infty\leq Ch\left(|\overline{q}|+{h}|\overline{p}|\right),
\end{equation*}
so that
\begin{eqnarray*}
F(p)&=&F\Big(\overline{p}+\int_0^h\nabla V(Q(t))dt\Big)
\\&\stackrel{\eqref{assumpt1},\eqref{assumpt4}}\leq& F(\overline{p})+C(|\overline{p}|+1)\left|\int_0^h\nabla V(Q(t))dt\right|+C\left(\int_0^h\nabla V(Q(t))dt\right)^2
\\&\leq& F(\overline{p})+Ch(|\overline{p}|+1)\left(|\overline{q}|+h|\overline{p}|\right)+Ch^2\left(|\overline{q}|+h|\overline{p}|\right)^2
\\&\leq& F(\overline{p})+Ch\big[N(\overline{q},\overline{p})+1\big].
\end{eqnarray*}
Therefore
\begin{align}
\F(\rho_*^h)&=\int_{\R^{2d}}F(p)\rho_*^h(q,p)dqdp=\int_{\R^{2d}}F(p)\rho_{k-1}^h(\overline{q},\overline{p})d\overline{q}d\overline{p}\nonumber
\\&\leq \int_{\R^{2d}}(F(\overline{p}) + ChN(\overline{q},\overline{p})+ Ch)\rho_{k-1}^h(\overline{q},\overline{p})d\overline{q}d\overline{p}\leq \F(\rho_{k-1}^h)+ChM_2(\rho_{k-1}^h)+Ch.\label{ineqFriction:schemes2ab}
\end{align}
For the entropy term, we have, since $\mathbf s_h$ is volume-preserving and bijective,
\begin{align}
S(\rho_*^h)&=\beta^{-1}\int_{\R^{2d}}\rho_*^h(q,p)\log \rho_*^h(q,p) dqdp=\beta^{-1}\int_{\R^{2d}}\rho_{k-1}^h(\mathbf{s}_h(q,p))\log\rho_{k-1}^h(\mathbf{s}_h(q,p))dqdp=S(\rho_{k-1}^h).\label{ineqEntropy:schemes2ab}
\end{align}
From (\ref{inequality:schemes2ab}), (\ref{ineqFriction:schemes2ab}), and (\ref{ineqEntropy:schemes2ab}), we obtain
\begin{equation*}
\widetilde{W}_h(\rho_{k-1}^h,\rho_k^h)\leq 2\gamma h(\A(\rho_{k-1}^h)-\A(\rho_k^h))+Ch^2M_2(\rho_{k-1}^h)+Ch^2.
\end{equation*}
Summing over $k=1$ to $n$ we obtain~\eqref{sumIneqn:scheme2a}.

For Scheme 2b, the equation~\eqref{eq:HamSys} only modifies slightly, in that the acceleration becomes constant:
\[
Q' = \frac Pm, \quad P' = -V(q).
\]
Similar estimates lead to the same result.

For Scheme 2c, the proof is again similar, by taking $\rho_*^h:=\mu_k^h$ and estimating the difference $\A(\mu_k^h)-\A(\rho_{k-1}^h)$ as is done above.
\end{proof}

\begin{lemma}\label{priorbound2}
There exist positive constants $T_0$, $h_0$, and $C$,  independent of the initial data, such that for any $0<h\leq h_0$, the solutions $\{\rho_k^h\}_{k\geq 1}$ for Scheme 2a, Scheme 2b, or Scheme 2c, satisfy
\begin{equation}
M_2(\rho_k^h)\leq C\big[M_2(\rho_0)+1\big] ~ \text{ and } ~|S(\rho_k^h)|\leq C\big[S(\rho_0) + M_2(\rho_0) + 1\big] ~ \text{for any }k\leq K_0,
\end{equation}
where $K_0=\lceil{T_0}/{h}\rceil$.
\end{lemma}

\begin{proof}
We detail the proof for Scheme 2a; the modifications for Schemes 2b and 2c are very minor.

For a fixed $i$, let $\widetilde P_i\in \Gamma(\rho_{i-1}^h,\rho_i^h)$ be the optimal plan in the definition of $ \widetilde W_h(\rho_{i-1}^h,\rho_i^h)$. We have
\begin{align}
&\left(\int_{\R^{2d}}|p|^2\rho_i^h(q,p)dqdp\right)^{\frac{1}{2}}=\left(\int_{\R^{4d}}|p'|^2\widetilde P_i^h(dqdpdq'dp')\right)^{\frac{1}{2}}\nonumber
\\&\qquad\leq\left(\int_{\R^{4d}}|p'-p|^2\widetilde P_i^h(dqdpdq'dp')\right)^{\frac{1}{2}}+\left(\int_{\R^{4d}}|p|^2\widetilde P_i^h(dqdpdq'dp')\right)^{\frac{1}{2}}\nonumber
\end{align}
By \eqref{ineq:pC}, we estimate
\begin{equation*}
\left(\int_{\R^{4d}}|p'-p|^2\widetilde P_i^h(dqdpdq'dp')\right)^{\frac{1}{2}}\leq C\widetilde {W}_h(\rho_{i-1}^h,\rho_i^h)^\frac{1}{2}+Ch\big[M_2(\rho_i^h)^\frac{1}{2}+M_2(\rho_{i-1}^h)^\frac{1}{2}\big],
\end{equation*}
and hence,
\begin{equation*}
\left(\int_{\R^{2d}}|p|^2\rho_i^h(q,p)dqdp\right)^{\frac{1}{2}}
\leq \left(\int_{\R^{2d}}|p|^2\rho_{i-1}^h(q,p)dqdp\right)^{\frac{1}{2}}
 + C\widetilde {W}_h(\rho_{i-1}^h,\rho_i^h)^{\frac{1}{2}}
 + Ch\big[M_2(\rho_i^h)^\frac{1}{2}+M_2(\rho_{i-1}^h)^\frac{1}{2}\big].\label{p-estimate1}
\end{equation*}
Summing over $i$ from $1$ to $k$ we obtain
\begin{align*}
\left(\int_{\R^{2d}}\abs{p}^2\rho_k^h(q,p)dqdp\right)^{\frac{1}{2}}&\leq C\sum_{i=1}^k\widetilde {W}_h(\rho_{i-1}^h,\rho_i^h)^{\frac{1}{2}}+Ch\sum_{i=1}^kM_2(\rho_{i-1}^k)^{\frac{1}{2}}+\left(\int_{\R^{2d}}\abs{p}^2\rho_0(q,p)dqdp\right)^{\frac{1}{2}}
\\&\leq C\sum_{i=1}^k\widetilde {W}_h(\mu_i^h,\rho_i^h)^{\frac{1}{2}}+Ch\sum_{i=1}^kM_2(\rho_i^k)^{\frac{1}{2}}+CM_2(\rho_0)^{\frac{1}{2}}.
\end{align*}
Therefore
\begin{align}
\int_{\R^{2d}}\abs{p}^2\rho_k^h(q,p)dqdp&\leq C\left(\sum_{i=1}^k\widetilde {W}_h(\mu_i^h,\rho_i^h)^{\frac{1}{2}}\right)^2 + Ch^2\left(\sum_{i=1}^kM_2(\rho_i^h)^{\frac{1}{2}}\right)^2+CM_2(\rho_0)\nonumber
\\&\leq Ck\sum_{i=1}^k\widetilde {W}_h(\mu_i^h,\rho_i^h)+Ckh^2\sum_{i=1}^kM_2(\rho_i^h)+CM_2(\rho_0). \label{M2p}
\end{align}
Similarly, we use~\eqref{est:qChproof} and the fact that
\begin{equation*}
q'=\frac{h}{2m\sqrt{3}}2\sqrt{3}\left(\frac{m}{h}(q'-q)-\frac{p+p'}{2}\right)+\frac{h}{2m}(p'+p)+q
\end{equation*}
to derive that
\begin{align*}
&\left(\int_{\R^{2d}}|q|^2\rho_i^h(q,p)dqdp\right)^{\frac{1}{2}}=\left(\int_{\R^{4d}}|q'|^2\widetilde P_i^h(dqdpdq'dp')\right)^{\frac{1}{2}}\nonumber
\\&\qquad\leq \frac{h}{2m\sqrt{3}}\left(\int_{\R^{4d}}12\left|\frac{m}{h}(q'-q)-\frac{p'+p}{2}\right|^2\widetilde P_i^h(dqdpdq'dp')\right)^\frac{1}{2}
 +\frac{h}{2m}\left(\int_{\R^{4d}}|p'|^2\widetilde P_i^h(dqdpdq'dp')\right)^\frac{1}{2}
\\&\qquad\qquad+\frac{h}{2m}\left(\int_{\R^{4d}}|p|^2\widetilde P_i^h(dqdpdq'dp')\right)^\frac{1}{2}+\left(\int_{\R^{2d}}|q|^2\rho_{i-1}^h(q,p)dqdp\right)^\frac{1}{2}
\\&\qquad\leq Ch\widetilde {W}_h(\rho_{i-1}^h,\rho_i^h)^\frac{1}{2}
+Ch\Big[M_2(\rho_{i-1}^h)^\frac{1}{2}+M_2(\rho_i^h)^\frac{1}{2}\Big]+\left(\int_{\R^{2d}}|q|^2\rho_{i-1}^h(q,p)dqdp\right)^\frac{1}{2}.
\end{align*}
Summing over $i$ from $1$ to $k$, we obtain
\begin{equation*}
\left(\int_{\R^{2d}}|q|^2\rho_k^h(q,p)dqdp\right)^{\frac{1}{2}}\leq Ch\sum_{i=1}^k\widetilde {W}_h(\rho_{i-1}^h,\rho_i^h)^\frac{1}{2}+Ch\sum_{i=1}^kM_2(\rho_i^h)^\frac{1}{2}+CM_2(\rho_0)^\frac{1}{2}
\end{equation*}
and therefore,
\begin{equation}
\int_{\R^{2d}}\gamma^2\abs{q}^2\rho_k^h(q,p)dqdp\leq Ckh^2\sum_{i=1}^k\widetilde {W}_h(\rho_{i-1}^h,\rho_i^h)+Ckh^2\sum_{i=1}^kM_2(\rho_i^h)+CM_2(\rho_0). \label{M2q}
\end{equation}
From (\ref{M2p}) and (\ref{M2q}) it holds that
\begin{align*}
M_2(\rho_k^h)&=\int_{\R^{2d}}(\abs{\gamma q}^2+\abs{p}^2)\rho_k^h(q,p)dqdp
\leq Ck\sum_{i=1}^k\widetilde {W}_h(\rho_{i-1}^h,\rho_i^h)+Ckh^2\sum_{i=1}^kM_2(\rho_i^h)+CM_2(\rho_0).
\end{align*}
Applying Lemma \ref{priorbound1:schemes2ab} with $n=k$, it follows that
\begin{align}
M_2(\rho_k^h)&\leq Ck\left[h(\A(\rho_0)-\A(\rho_k^h))+Ch^2\sum_{i=0}^{k}M_2(\rho_i^h)+Ckh^2\right]+Ckh^2\sum_{i=1}^kM_2(\rho_i^h)+CM_2(\rho_0)\nonumber
\\&\leq -CkhS(\rho_k^h)+Ckh^2\sum_{i=1}^kM_2(\rho_i^k)
+CM_2(\rho_0)
+Ckh\A(\rho_0)
+Ck^2h^2.\label{M2}
\end{align}
By inequality (29) in \cite{JKO98},  $S(\rho_k^h)$ is bounded from below by $M_2(\rho_k^h)$,
\begin{equation}
S(\rho_k^h)\geq -C- CM_2(\rho_k^h).\label{EntropyBoundBelow}
\end{equation}
Substituting \eqref{EntropyBoundBelow} into~\eqref{M2} we have
\begin{equation}
M_2(\rho_k^h)\leq C_1^2kh^2\sum_{i=1}^kM_2(\rho_i^k)+ C_1khM_2(\rho_k^h)+C_1(k^2h^2+1)
+ C_1 M_2(\rho_0),\label{M2estimate}
\end{equation}
where we fix the constant $C_1$, and use it to set the time horizon $T_0$:
\begin{equation}
T_0=\frac{1}{4C_1},~~K_0=\left\lceil\frac{T_0}{h}\right\rceil.
\end{equation}
We emphasize that $C_1$, and hence $T_0$, is independent of the initial data. We now choose $h_0\leq T_0$ so small that for all $h\leq h_0$ we have $K_0h\leq 2T_0$ and $C_1K_0h\leq\frac{1}{2}$. Then it follows from~\eqref{M2estimate} that, for any $h\leq h_0, k\leq K_0$,
\begin{equation}
\frac34M_2(\rho_k^h)\leq C_1^2kh^2\sum_{i=1}^kM_2(\rho_i^h)+C_1(4T_0^2+1)+ C_1 M_2(\rho_0).\label{M2sumestimate}
\end{equation}
Hence
\begin{align}
\frac34\sum_{i=1}^{K_0}M_2(\rho_i^h)&\leq C_1^2K_0^2h^2\sum_{i=1}^{K_0}M_2(\rho_i^h)+K_0(T_0 + C_1)+ C_1 M_2(\rho_0)\nonumber
\\&\leq4 C_1^2T_0^2\sum_{i=1}^{K_0}M_2(\rho_i^h)+K_0(T_0+C_1)+ C_1 M_2(\rho_0)\\
&\leq \frac14\sum_{i=1}^{K_0}M_2(\rho_i^h)+K_0(T_0+C_1)+ C_1 M_2(\rho_0).\nonumber
\end{align}
Consequently,
\begin{equation}
\sum_{i=1}^{K_0}M_2(\rho_i^h)\leq 2K_0(T_0+C_1)+ 2C_1 M_2(\rho_0). \label{M2sumestimatefinal}
\end{equation}
Substituting~\eqref{M2sumestimatefinal} into~\eqref{M2sumestimate}, we obtain
\begin{equation}
M_2(\rho_k^h)\leq \frac23\Big(2 +K_0\Big) (T_0+C_1) + C_1 M_2(\rho_0). \label{M2BoundAbove}
\end{equation}
This finishes the proof of the boundedness of $M_2(\rho_k^h)$.

We now show that the entropy $S(\rho_k^h)$ is also bounded.
From~\eqref{EntropyBoundBelow} and~\eqref{M2BoundAbove}, it follows that $S(\rho_k^h)$ is bounded from below.
It remains to find an upper bound. Applying Lemma~\ref{priorbound1:schemes2ab} for $n=k$, and noting that  $F(\rho_k^h)\geq 0$, $\widetilde W_h(\rho_{i-1}^h,\rho_i^h)\geq 0$ for all $i$, we have
\begin{align}
S(\rho_k^h)&\leq \A(\rho_0)+Ch\sum_{i=0}^kM_2(\rho_i^h)+Ckh
\leq Ch\sum_{i=1}^kM_2(\rho_i^h)+C\big[S(\rho_0)+M_2(\rho_0)\big]+2CT_0.\label{EntropyBoundABOVE}
\end{align}
By combining with~\eqref{M2sumestimatefinal} we obtain the upper bound for the entropy.
This completes the proof of the lemma.
\end{proof}

The following lemma extends Lemma \ref{priorbound2} to any $T>0$. The proof is the same as Lemma 5.3 in \cite{Hua00}, and we omit it.
\begin{lemma}
\label{lemma:priorbound3:Scheme2a}
Let $\{\rho_k^h\}_{k\geq1}$ be the sequence of the minimizers of Scheme 2a or Scheme 2b for fixed $h>0$. For any $T>0$, there exists a constant $C>0$ depending  on $T$ and on the initial data such that
\begin{equation}
M_2(\rho_k^h)\leq C,\label{2moment}
\end{equation}
\begin{equation}
\sum_{i=1}^k W_h^*(\rho_{i-1}^h,\rho_i^h)\leq Ch, \label{summetric}
\end{equation}
\begin{equation}
\int_{\R^{2d}}\max\{\rho_k^h\log \rho_k^h, 0\}\,dqdp\leq C, \label{absoluteentropy}
\end{equation}
for any $h\leq h_0$ and $k\leq K_h$, where
\begin{equation*}
K_h=\left\lceil\frac{T}{h}\right\rceil.
\end{equation*}
For Scheme 2c the same inequalities hold, with~\eqref{summetric} replaced by
\begin{equation*}
\sum_{i=1}^k W_h(\mu_i^h,\rho_i^h)\leq Ch. \label{summetric-scheme2c}
\end{equation*}
\end{lemma}

\section{Proof of Theorem~\ref{theo:maintheorem}}
\label{section:Convergence}

In this section we bring all the parts together to prove Theorem~\ref{theo:maintheorem}. The structure of this proof  is the same as that of e.g.~\cite{JKO98,Hua00}, and we refer to those references for the parts that are very similar. The main difference lies in the convergence of the discrete Euler-Lagrange equations for each of the cases to the weak formulation of the Kramers equation as $h\to0$.

Throughout we fix $T>0$ and for each $h>0$ we set
\[
K_h \colonequals \lceil T/h\rceil.
\]
The proof of the space-time weak compactness~\eqref{weaklyconverence} is the same for the three schemes. Let $(\rho_k^h)_k$ be the sequence of minimizers constructed by any of the three schemes, and let $t\mapsto \rho^h(t)$ be the piecewise-constant interpolation~\eqref{interpolation}. By Lemma~\ref{lemma:priorbound3:Scheme2a} we have
\begin{align}
M_2(\rho^h(t))+\int_{\R^{2d}}\max\{\rho^h(t)\log \rho^h(t), 0\}\,dqdp&\leq C,\qquad\text{for all} \quad 0\leq t \leq T. \label{sumM2Entropy}
\end{align}
Since the function $z\mapsto \max\{z\log z,0\}$ has super-linear growth, \eqref{sumM2Entropy} guarantees that there exists a subsequence, denoted again by $\rho^h$, and a function $\rho\in L^1((0,T)\times\R^{2d})$ such that
\begin{equation}
\rho^h\rightarrow \rho ~\text{ weakly in }~ L^1((0,T)\times\R^{2d}).
\end{equation}
This proves~\eqref{weaklyconverence}.

The proof of the stronger convergence~\eqref{pointwiseconvergence} and of the continuity~\eqref{intinialconvergence} at $t=0$ follows the same lines as in~\cite{JKO98,Hua00}. The main estimate is the `equi-near-continuity' estimate
\[
d\big(\rho^h(t_1),\rho^h(t_2)\big)^2 \leq C(|t_2-t_1| + h),
\]
where $d(\rho_0,\rho_1)$ is the metric generated by the quadratic cost $|q-q'|^2 + |p-p'|^2$. This estimate follows from the inequality (see~\eqref{ineqs:qpC})
\[
|q-q'|^2 + |p-p'|^2 \leq C\big[C_h^*(q,p;q',p') + h^2 N(q,p) + h^2 N(q',p')\big],
\]
and the estimates~\eqref{sumM2Entropy} and~\eqref{summetric}; see~\cite[Theorem~5.2]{Hua00}.

The only remaining statement of Theorem~\ref{theo:maintheorem} is the characterization of the limit in terms of the solution of the Kramers equation, and we now describe this.

Let $\rho^h$ be generated by one of the three schemes. We now prove that the limit $\rho$ satisfies the weak version of the Kramers equation~\eqref{weakKReqn}.
Fix $T>0$ and $\varphi\in C_c^\infty((-\infty,T)\times \R^{2d})$; all constants~$C$ below depend on the parameters of the problem, on the initial datum $\rho_0$, and on~$\varphi$, but are independent of $k$ and of $h$. We first discuss Schemes 2a and 2b.

Let $P_k^{h*}\in\Gamma(\rho_{k-1}^h,\rho_k^h)$ be the optimal plan for $W_h^*(\rho_{k-1}^h,\rho_k^h)$, where the star indicates the quantities associated with either Scheme 2a or Scheme 2b.
For any $0<t<T$, we have
\begin{align}
&\int_{\R^{2d}}\big[\rho_k^h(q,p)-\rho_{k-1}^h(q,p)\big]\,\varphi(t,q,p)dqdp \notag\\
&\qquad=\int_{\R^{2d}}\rho_k^h(q',p')\varphi(t,q',p')dq'dp'-\int_{\R^{2d}}\rho_{k-1}^h(q,p)\varphi(t,q,p)dqdp\nonumber
\\
&\qquad=\int_{\R^{4d}}\big[\varphi(t,q',p')-\varphi(t,q,p)\big]\,P_k^{h*}(dqdpdq'dp') \nonumber
\\&\qquad=\int_{\R^{4d}}\big[(q'-q)\cdot\nabla_{q'}\varphi(t,q',p')+(p'-p)\cdot\nabla_{p'}\varphi(t,q',p')\big]P_k^{h*}(dqdpdq'dp') +\varepsilon_k,\label{timederivativeappro}
\end{align}
where
\begin{eqnarray}
|\varepsilon_k|&\leq& C\int_{\R^{4d}}\big[|q'-q|^2+|p'-p|^2\big]\,P_k^{h*}(dqdpdq'dp')\notag\\
&\stackrel{\eqref{ineqs:qpC}}\leq& C W^*_h(\rho_{k-1}^h,\rho_k^h) + Ch^2\big[M_2(\rho_{k-1}^h) + M_2(\rho_k^h)\big]\notag\\
&\stackrel{\eqref{sumM2Entropy}}\leq& C W^*_h(\rho_{k-1}^h,\rho_k^h) + Ch^2.
\label{epsilon}
\end{eqnarray}
By combining~\eqref{timederivativeappro} with~\eqref{eq:realcostEL} we find
\begin{multline}
\int_{\R^{2d}}\left(\frac{\rho_k^h(t,q,p)-\rho_{k-1}^h(q,p)}{h}\right)\varphi(t,q,p)dqdp
\\
=\int_{\R^{2d}}\left[\frac{p}{m}\cdot\nabla_q\varphi(t,q,p)-(\nabla V(q)+\gamma\nabla F(p))\cdot \nabla_p\varphi(t,q,p)+\gamma \kT \Delta_p\varphi(t,q,p)\right]\rho_k^h(q,p)dqdp
\\
+\theta_k(t), \label{approximation1}
\end{multline}
where
\begin{eqnarray}
|\theta_k(t)|&\leq&\frac{|\varepsilon_k|}{h}+Ch\big[W_h^*(\rho_{k-1}^h,\rho_k^h) + M_2(\rho_{k-1}^h) + M_2(\rho_k^h)+1\Big]\notag\\
&\stackrel{\eqref{sumM2Entropy},\eqref{epsilon}}\leq &
\frac Ch W_h^*(\rho_{k-1}^h,\rho_k^h)  + Ch.
\label{thetaform}
\end{eqnarray}
Note that $\theta_k$ depends on $t$ through the $t$-dependence of $\varphi$.
Next, from (\ref{approximation1}), for $k\geq 1$ we have
\begin{align}
&\int_{(k-1)h}^{kh}\int_{\R^{2d}}\left(\frac{\rho_k^h(q,p)-\rho_{k-1}^h(q,p)}{h}\right)\varphi(t,q,p)dqdpdt\nonumber
\\&\qquad=\int_{(k-1)h}^{kh}\int_{\R^{2d}}\left[\frac{p}{m}\cdot\nabla_q\varphi(t,q,p)-(\nabla V(q)+\gamma\nabla F(p))\cdot \nabla_p\varphi(t,q,p)+\gamma \kT \Delta_p\varphi(t,q,p)\right]\rho_k^h(q,p)dqdpdt\nonumber
\\&\qquad\qquad+ \int_{(k-1)h}^{kh}\theta_k(t)dt\nonumber
\\&\qquad=\int_{(k-1)h}^{kh}\int_{\R^{2d}}\left[\frac{p}{m}\cdot\nabla_q\varphi(t,q,p)-(\nabla V(q)+\gamma\nabla F(p))\cdot \nabla_p\varphi(t,q,p)+\gamma \kT \Delta_p\varphi(t,q,p)\right]\rho^h(t,q,p)dqdpdt\nonumber
\\&\qquad\qquad+ \int_{(k-1)h}^{kh}\theta_k(t)dt.\nonumber
\end{align}
Summing from $k=1$ to $K_h$ we obtain
\begin{align}
&\sum_{k=1}^{K_h}\int_{(k-1)h}^{kh}\int_{\R^{2d}}\left(\frac{\rho_k^h(q,p)-\rho_{k-1}^h(q,p)}{h}\right)\varphi(t,q,p)dqdpdt\nonumber
\\&\qquad=\int_0^T\int_{\R^{2d}}\left[\frac{p}{m}\cdot\nabla_q\varphi(t,q,p)-(\nabla V(q)+\gamma\nabla F(p))\cdot \nabla_p\varphi(t,q,p)+\gamma \kT \Delta_p\varphi(t,q,p)\right]\rho^h(t,q,p)dqdpdt\nonumber
\\&\qquad\qquad+ R_h,\label{approximationsum}
\end{align}
where
\begin{equation}
R_h=\sum_{k=1}^{K_h}\int_{(k-1)h}^{kh}\theta_k(t)dt.\label{Rh}
\end{equation}
By a discrete integration by parts, we can rewrite the left hand side of~\eqref{approximationsum} as
\begin{align}
&-\int_0^h\int_{\R^{2d}}\rho_0(q,p)\frac{\varphi(t,q,p)}{h}dqdpdt+\int_0^T\int_{\R^{2d}}\rho^h(t,q,p)\left(\frac{\varphi(t,q,p)-\varphi(t+h,q,p)}{h}\right)dqdpdt\label{approximationsumFINAL}.
\end{align}
From~\eqref{approximationsum} and~\eqref{approximationsumFINAL} we obtain
\begin{align}
&\int_0^T\int_{\R^{2d}}\rho^h(t,q,p)\left(\frac{\varphi(t,q,p)-\varphi(t+h,q,p)}{h}\right)dqdpdt\nonumber
\\&=\int_0^T\int_{\R^{2d}}\left[\frac{p}{m}\cdot\nabla_q\varphi(t,q,p)-(\nabla V(q)+\gamma\nabla F(p))\cdot \nabla_p\varphi(t,q,p)+\gamma \kT \Delta_p\varphi(t,q,p)\right]\rho^h(t,q,p)dqdpdt\nonumber
\\&\qquad+\int_0^h\int_{\R^{2d}}\rho_0(q,p)\frac{\varphi(t,q,p)}{h}dqdpdt+R_h. \label{weakconvergence}
\end{align}
Now $R_h\rightarrow 0$ as $h\to0$, since
\begin{eqnarray}
|R_h|\stackrel{\eqref{Rh}}\leq \sum_{k=1}^{K_h}\int_{(k-1)h}^{kh}|\theta_k(t)|dt&\stackrel{\eqref{thetaform}}\leq& C\sum_{k=1}^{K_h}\int_{(k-1)h}^{kh}\left(\frac{1}{h}W_h^*(\rho_{k-1}^h,\rho_k^h)+h\right)dt\nonumber
\\&=& C\sum_{k=1}^{K_h}\big[W_h^*(\rho_{k-1}^h,\rho_k^h)+Ch^2\big]
\stackrel{\eqref{summetric}}\leq Ch.\nonumber
\end{eqnarray}
Taking the limit $h\rightarrow0$ in~\eqref{weakconvergence} yields equation~\eqref{weakKReqn}.

\bigskip

For Scheme 2c, only~\eqref{timederivativeappro} is different:
\begin{align*}
&\int_{\R^{2d}}\big[\rho_k^h(q,p)-\rho_{k-1}^h(q,p)\big]\,\varphi(t,q,p)\,dqdp
\\&\qquad=\int_{\R^{2d}}\rho_k^h(q',p')\varphi(t,q',p')dq'dp'-\int_{\R^{2d}}\rho_{k-1}^h(q,p))\varphi(t,q,p)dqdp\nonumber
\\&\qquad=\int_{\R^{2d}}\rho_k^h(q',p')\varphi(t,q',p')dq'dp'-\int_{\R^{2d}}\mu_k^h(q,p)\varphi(t,\sigma_h(q,p))dqdp\nonumber
\\&\qquad=\int_{\R^{4d}}\Big[\varphi(t,q',p')-\varphi\Big(t,q-\frac{p}{m}h,p+\nabla V(q)h\Big)\Big]\,P_k^h(dqdpdq'dp') \nonumber
\\&\qquad=\int_{\R^{4d}}\left[(q'-q+\frac{p}{m}h)\cdot\nabla_{q'}\varphi(t,q',p')+(p'-p-\nabla V(q)h)\cdot\nabla_{p'}\varphi(t,q',p')\right]P_k^h(dqdpdq'dp') +\varepsilon_k,\nonumber
\end{align*}
where
\begin{equation*}
|\varepsilon_k|\leq C\int_{\R^{4d}}\left(\gamma^2\left|q'-q+\frac{p}{m}h\right|^2+|p'-p-\nabla V(q)h|^2\right)P_k^h(dqdpdq'dp')
\end{equation*}
with the constant $C$ depending only on $\varphi$.
Since $\abs{p'-p}^2, \abs{q'-q}^2\leq CC_h(q,p;q',p')$ and $\abs{\nabla V(q)}^2\leq C\abs{q}^2$,
\begin{align*}
\gamma^2\left|q'-q+\frac{p}{m}h\right|^2+|p'-p-h\nabla V(q)|^2&\leq 2\left(\gamma^2|q-q'|^2+\frac{\gamma^2h^2}{m^2}|p|^2+|p-p'|^2+h^2|\nabla V(q)|^2\right)
\\&\leq CC_h(q,p;q',p')+Ch^2N(q,p).
\end{align*}
Therefore
\begin{align*}
|\varepsilon_k|&\leq C\int_{\R^{4d}}\left[C_h(q,p;q',p')+h^2N(q,p)+h^2\right]P_k^h(dqdpdq'dp')\nonumber
\\&=CW_h(\mu_k^h,\rho_k^h)+CM_2(\mu_k^h)h^2+Ch^2\nonumber
\\&\leq CW_h(\mu_k^h,\rho_k^h)+ Ch^2.\label{varepsilon}
\end{align*}
The rest of the proof is the same.

\begin{center}
\textbf{Acknowledgement}
\end{center}
The research of the paper has received funding from the ITN ``FIRST" of the Seventh Framework Programme of the European Community (grant agreement number 238702).
\bibliography{DPZdiscreteKramer}

\newcommand{\etalchar}[1]{$^{#1}$}
\begin{thebibliography}{AMP{\etalchar{+}}12}

\bibitem[ADPZ11]{ADPZ11}
S.~Adams, N.~Dirr, M.~A. Peletier, and J.~Zimmer.
\newblock From a large-deviations principle to the {W}asserstein gradient flow:
  a new micro-macro passage.
\newblock {\em Communications in Mathematical Physics}, 307:791--815, 2011.

\bibitem[AG08]{AG08}
L.~Ambrosio and W.~Gangbo.
\newblock Hamiltonian {ODE}s in the {W}asserstein space of probability
  measures.
\newblock {\em Comm. Pure Appl. Math.}, 61(1):18--53, 2008.

\bibitem[AGS08]{AGS08}
L.~Ambrosio, N.~Gigli, and G.~Savar\'e.
\newblock {\em Gradient flows in metric spaces and in the space of probability
  measures}.
\newblock Lectures in Mathematics. ETH Z{\"u}rich. Birkhauser, Basel, 2nd
  edition, 2008.

\bibitem[AMP{\etalchar{+}}12]{ArnrichMielkePeletierSavareVeneroni11TR}
S.~Arnrich, A.~Mielke, M.~Peletier, G.~Savar{\'e}, and M.~Veneroni.
\newblock Passing to the limit in a {W}asserstein gradient flow: From diffusion
  to reaction.
\newblock {\em Calculus of Variations and Partial Differential Equations},
  44:419--454, 2012.

\bibitem[Bro28]{Brown}
R.~Brown.
\newblock A brief account of microscopical observations made in the months of
  june, july and august, 1827, on the particles contained in the pollen of
  plants; and on the general existence of active molecules in organic and
  inorganic bodies.
\newblock Privately circulated in 1827. Reprinted in the Edinburgh new
  Philosophical Journal (pp. 358-371, July-September, 1828).

\bibitem[Cha03]{Chavanis03}
P.~H. Chavanis.
\newblock Generalized thermodynamics and fokker-planck equations: Applications
  to stellar dynamics and two-dimensional turbulence.
\newblock {\em Phys. Rev. E}, 68:036108, Sep 2003.

\bibitem[CLL04]{Chavanis04}
P.~H. Chavanis, P.~Lauren\c{c}ot, and M.~Lemou.
\newblock Chapman-{E}nskog derivation of the generalized smoluchowski equation.
\newblock {\em Physica A: Statistical Mechanics and its Applications},
  341:145--164, October 2004.

\bibitem[CMV03]{CarrilloMcCannVillani03}
J.~A. Carrillo, R.~J. McCann, and C.~Villani.
\newblock Kinetic equilibration rates for granular media and related equations:
  entropy dissipation and mass transportation estimates.
\newblock {\em Rev. Mat. Iberoamericana}, 19(3):971--1018, 2003.

\bibitem[CSR96]{CSR96}
P.~H. Chavanis, J.~Sommeria, and R.~Robert.
\newblock Statistical mechanics of two dimensional vortices and collisionless
  stellar systems.
\newblock {\em The Astrophysical Journal}, 471:385--399, 1996.

\bibitem[DLR12]{DuongLaschosRenger12TR}
M.~H. Duong, V.~Laschos, and D.R.M. Renger.
\newblock Wasserstein gradient flows from large deviations of thermodynamic
  limits (submitted).
\newblock \url{http://arxiv.org/abs/1203.0676}, 2012.

\bibitem[DLZ11]{Laschos2011}
N.~Dirr, V.~Laschos, and J.~Zimmer.
\newblock Upscaling from particle models to entropic gradient flows.
\newblock {\em To appear in J. Math. Phys}, 2011.

\bibitem[DM10]{DelarueMenozzi10}
F.~Delarue and S.~Menozzi.
\newblock Density estimates for a random noise propagating through a chain of
  differential equations.
\newblock {\em J. Funct. Anal.}, 259(6):1577--1630, 2010.

\bibitem[DMM10]{DuringMatthesMilisic10}
B.~D{\"u}ring, D.~Matthes, and J.~Mili{\v{s}}i{\'c}.
\newblock A gradient flow scheme for nonlinear fourth order equations.
\newblock {\em Discrete Contin. Dyn. Syst. Ser. B}, 14(3):935--959, 2010.

\bibitem[DZ87]{DemboZeitouni98}
A.~Dembo and O.~Zeitouni.
\newblock {\em Large deviations techniques and applications}, volume~38 of {\em
  Stochastic modelling and applied probability}.
\newblock Springer, New York, NY, USA, 2nd edition, 1987.

\bibitem[GW09]{GW09}
W.~Gangbo and M.~Westdickenberg.
\newblock Optimal transport for the system of isentropic {E}uler equations.
\newblock {\em Comm. Partial Differential Equations}, 34(7-9):1041--1073, 2009.

\bibitem[Hua00]{Hua00}
C.~Huang.
\newblock A variational principle for the {K}ramers equation with unbounded
  external forces.
\newblock {\em J. Math. Anal. Appl.}, 250(1):333--367, 2000.

\bibitem[Hua11]{Hua11}
C.~Huang.
\newblock A variational principle for a class of ultraparabolic equations.
\newblock 2011.

\bibitem[JKO97]{JordanKinderlehrerOtto97}
R.~Jordan, D.~Kinderlehrer, and F.~Otto.
\newblock Free energy and the {F}okker-{P}lanck equation.
\newblock {\em Phys. D}, 107(2-4):265--271, 1997.
\newblock Landscape paradigms in physics and biology (Los Alamos, NM, 1996).

\bibitem[JKO98]{JKO98}
R.~Jordan, D.~Kinderlehrer, and F.~Otto.
\newblock The variational formulation of the fokker-planck equation.
\newblock {\em SIAM Journal on Mathematical Analysis}, 29(1):1--17, 1998.

\bibitem[Kra40]{Kramers40}
H.~A. Kramers.
\newblock Brownian motion in a field of force and the diffusion model of
  chemical reactions.
\newblock {\em Physica}, 7:284--304, 1940.

\bibitem[Mie05]{Mielke05a}
A.~Mielke.
\newblock Evolution of rate-independent systems.
\newblock In {\em Evolutionary equations. {V}ol. {II}}, Handb. Differ. Equ.,
  pages 461--559. Elsevier/North-Holland, Amsterdam, 2005.

\bibitem[MTL02]{MielkeTheilLevitas02}
A.~Mielke, F.~Theil, and V.~I. Levitas.
\newblock A variational formulation of rate-independent phase transformations
  using an extremum principle.
\newblock {\em Arch. Ration. Mech. Anal.}, 162(2):137--177, 2002.

\bibitem[{\"{O}}tt05]{Oettinger05}
H.~C. {\"{O}}ttinger.
\newblock {\em Beyond equilibrium thermodynamics}.
\newblock Wiley-Interscience, 1st edition, 2005.

\bibitem[PR11]{Peletier2011}
M.A. Peletier and D.R.M. Renger.
\newblock Variational formulation of the {F}okker-{P}lanck equation with decay:
  a particle approach (submitted).
\newblock \url{http://arxiv.org/abs/1108.3181}, 2011.

\bibitem[SS04]{SandierSerfaty04}
E.~Sandier and S.~Serfaty.
\newblock {Gamma-convergence of gradient flows with applications to
  Ginzburg-Landau}.
\newblock {\em Communications on Pure and Applied Mathematics},
  57(12):1627--1672, 2004.

\bibitem[Ste08]{Stefanelli08}
U.~Stefanelli.
\newblock {The Brezis--Ekeland principle for doubly nonlinear equations}.
\newblock {\em SIAM Journal on Control and Optimization}, 47:1615, 2008.

\bibitem[Vil03]{Vil03}
C.~Villani.
\newblock {\em Topics in optimal transportation}, volume~58 of {\em Graduate
  Studies in Mathematics}.
\newblock American Mathematical Society, Providence, RI, 2003.

\bibitem[vR11]{VonRenesse11}
M.-K. von Renesse.
\newblock On optimal transport view on {S}chr{\"o}dinger's equation.
\newblock {\em To appear in Canad. Math. Bull.}, 2011.

\bibitem[Wes10]{Westdickenberg10}
M.~Westdickenberg.
\newblock Projections onto the cone of optimal transport maps and compressible
  fluid flows.
\newblock {\em J. Hyperbolic Differ. Equ.}, 7(4):605--649, 2010.

\end{thebibliography}
\bibliographystyle{alpha}
\end{document}